\documentclass[12pt,twoside]{preprint}
\usepackage[T1]{fontenc}
\usepackage[margin=1.3in]{geometry}
\usepackage{amssymb}
\usepackage{amsmath}
\usepackage{hyperref}
\usepackage{breakurl}
\usepackage{mhenvs}
\usepackage{mhequ}
\usepackage{booktabs}
\usepackage{tikz} 
\usepackage{mathrsfs}
\usepackage{times}
\usepackage{microtype}
\usepackage{wasysym}
\usepackage{centernot}
\usepackage{appendix}



\newcommand{\eqdef}{\stackrel{\mbox{\tiny def}}{=}}

\newcommand{\C}{\mathbb{C}}
\newcommand{\E}{\mathbb{E}}
\newcommand{\N}{\mathbb{N}}
\renewcommand{\P}{{\mathbb P}}

\newcommand{\R}{\mathbb{R}}
\newcommand{\T}{\mathbb{T}}
\newcommand{\Z}{\mathbb{Z}}

\newcommand{\nn}{\mathfrak{m}}

\newcommand{\KK}{\mathfrak{K}}

\newcommand{\Bb}{\mathcal{B}}
\newcommand{\Cc}{\mathcal{C}}
\newcommand{\Dd}{\mathcal{D}}

\newcommand{\ga}{\gamma}

\newcommand{\eps}{\varepsilon}
\let\epsilon\varepsilon
\newcommand{\ka}{\kappa}

\renewcommand{\subset}{\subseteq}

\newcommand{\msf}{\mathsf}

\newcommand{\ang}[1]{\left\langle #1 \right\rangle} 
\newcommand{\Vg}{V_{\gamma}} 								
\newcommand{\Vng}{V_{\gamma}^0} 								
\newcommand{\norm}[2]{\left\lvert #2 \right\rvert_{#1}}
\newcommand{\Norm}[2]{\left\Vert #2 \right\Vert_{#1}}
\newcommand{\bigNorm}[2]{\big\Vert #2 \big\Vert_{#1}}

\newcommand{\supp}{\mathrm{supp}}

\newcommand{\Oo}{\mathcal{O}}								
\newcommand{\Deltag}{\Delta_{\gamma}}						
\newcommand{\Hgbetab}{\mathcal{H}^{\gamma}_{\beta,b} }	
\newcommand{\Pgbeta}{\mathbf{P}^{\gamma}_{\beta,0} }		
\newcommand{\Pgbetab}{\mathbf{P}^{\gamma}_{\beta,b} }		
\newcommand{\Pgbetabg}{\mathbf{P}^{\gamma}_{\beta,b_\gamma} }		
\newcommand{\Egbeta}{\mathbf{E}^{\gamma}_{\beta,0} }		
\newcommand{\Egbetab}{\mathbf{E}^{\gamma}_{\beta,b} }		
\newcommand{\Egbetabg}{\mathbf{E}^{\gamma}_{\beta,b_\gamma} }		
\newcommand{\etaN}{\eta^N}		

\newcommand{\hg}{h_{\ga}} 								
\newcommand{\kg}{\ka_\ga} 								
\newcommand{\Kg}{K_\ga} 								
\newcommand{\LN}{\Lambda_N} 							
\newcommand{\Le}{\Lambda_\eps}		 					
\newcommand{\Ze}{\Z_\eps}		 					
\newcommand{\SN}{\Sigma_N} 							
\newcommand{\cg}{c_{\ga}}		 	 					
\newcommand{\Mg}{M_\ga}  								
\newcommand{\Xg}{X_{\ga}}		  						
\newcommand{\Xng}{X_\ga^0} 							
\newcommand{\Xn}{X^0} 							
\newcommand{\CGG}{\mathfrak{c}_{\ga}} 					

\newcommand{\Zg}{Z_{\ga}}  								

\newcommand{\taun}{\tau_{\ga,\nn}}							 

\newcommand{\CG}{\mathfrak{c}_\ga} 						
\newcommand{\Ex}{\msf{Ext}} 								
\newcommand{\Err}{\msf{Err}} 								

\newcommand{\hKg}{\hat{K}_{\ga}} 							

\renewcommand{\L}{\mathscr{L}}

\begin{document}

\title{Tightness of the Ising-Kac model on the two-dimensional torus}
\author{Martin Hairer$^1$ and Massimo Iberti$^2$}
\institute{Imperial College London, UK, \email{m.hairer@imperial.ac.uk} \and University of Warwick, UK, \email{m.iberti@warwick.ac.uk}}

\maketitle

\begin{abstract}
We consider the sequence of Gibbs measures of Ising models with Kac interaction defined on a periodic two-dimensional discrete torus near criticality.
Using the convergence of the Glauber dynamic proven by H. Weber and J.C. Mourrat \cite{MourratWeber} and a method by H. Weber and P. Tsatsoulis employed in \cite{tsatsoulis2016spectral},
we show tightness for the sequence of Gibbs measures of the Ising-Kac model near criticality and characterise the law of the limit as the $\Phi^4_2$ measure on the torus.

Our result is very similar to the one obtained by M.~Cassandro, R.~Marra, E.~Presutti \cite{cassandro1995corrections} on $\Z^2$, but our strategy takes advantage of the dynamic, instead of correlation inequalities.
In particular, our result covers the whole critical regime and does not require the large temperature / large mass / small coupling assumption present in earlier results.
\end{abstract}

\tableofcontents

\section{Introduction}
\label{sec:intro}
Let $N>0$ be a positive integer and consider the (periodic) lattice $\LN = \{1-N,\dots,N\}^2$. For $\gamma > 0$, let $\KK: \R^2 \to [0,1]$ be a twice differentiable, non negative, isotropic function supported on a ball of radius $3$ and define $\kg$ for $x \in \LN$
\begin{equation}
\label{e:samplek}
\kg(x) \simeq \gamma^2 \KK(\gamma|x|)\ ,\qquad \sum_{z \in \LN \setminus \{0\}} \kg(z) = 1\:.
\end{equation}
Consider a spin system formed by a set of spins parametrized by the lattice $\LN$. Each spin can assume the value $+1$ or $-1$ representing two possible states of the magnetization and we will denote with $\SN=\{-1,1\}^{\LN}$ the set of all possible configurations. The Ising-Kac model on the two-dimensional lattice with periodic boundary condition and external magnetization $b \in \R$ is given by the following Hamiltonian 
\begin{equation}
\label{e:hamiltonian}
\Hgbetab(\sigma) = \frac{\beta}{2} \sum_{x, y \in \LN} \kg(x-y) \sigma_x \sigma_y + b \sum_{x \in \LN} \sigma_x\, \qquad \sigma \in \SN \;.
\end{equation}
The Gibbs measure over $\Sigma_N$ associated to the potential \eqref{e:samplek}, with inverse temperature $\beta$ and external magnetic field $b$, is given by
\begin{equation}
\label{e:gibbsdef}
\Pgbetab (\sigma)\eqdef \left( Z_{\gamma,\beta,b}^N\right)^{-1} \exp\left( \Hgbetab(\sigma)\right)\;,
\end{equation}
where $Z_{\gamma,\beta,b}^N$ is the partition function that makes \eqref{e:gibbsdef} a probability measure. We will also denote with $\Egbetab$ the expectation under $\Pgbetab$. For technical reason, we set $\kg(0) = 0$ and we remark that its precise value doesn't effect \eqref{e:gibbsdef}.

As in \cite{MourratWeber}, we will let the inverse temperature $\beta$ converge in a precise way as $\gamma \to 0$ to $\beta_c=1$ the critical value of the mean-field system.
The purpose of this paper is to prove the tightness of the magnetisation fluctuation field
\begin{equation}
\label{e:discretefluctuationb}
\gamma^{-1} \left(\sigma_{\lfloor N \cdot \rfloor} - \Egbetab\left[\sigma_{\lfloor N \cdot \rfloor}\right] \right)
\end{equation}
in the strong topology of $\mathcal{S}'(\T^2)$ as both $N \to \infty$ and $\gamma \to 0$ in a precise way that we shall describe later.

Moreover, in case $b=0$, we are also able to characterise the limit as the $\Phi^4(\T^2)$ measure, formally described by
\begin{equation}
\label{e:Phi42measure}
\mathcal{Z}^{-1}\exp\left(-\int_{\T^2} \frac{1}{2} |\nabla \Phi(x)|^2 + \frac{1}{12}\Phi^{:4:}(x) -  \frac{A}{2}\Phi^{:2:}(x)\ dx\right) d \Phi\;,
\end{equation}
where $\Phi^{:4:}$ denotes the Wick renormalisation of the fourth power of the field. For a more detailed and formal definition, see for example \cite{Nelson,GlimmJaffe}.

The Ising-Kac model is a mean-field model with ferromagnetic long range potential that has been introduced in statistical mechanics for its simplicity and because it provides a framework to recover rigorously the van der Waals theory \cite{KUH1963vanderWaals} of phase transition. It has been then developed by Lebowitz and Penrose in \cite{LebowitzPenrose1966}, see also \cite{presutti2008scaling} for more details.
The model has already been useful to study the $\Phi^4_d$ theory, see \cite{Gawedzki1985}, where a renormalisation group approach has been used to approximate $\Phi^4_d$ with generalised Ising models, and \cite{Simon1973} with classical Ising spins. 

The present work is mainly built upon the article of Weber and Mourrat \cite{MourratWeber}, where the Glauber dynamic on a periodic two-dimensional lattice is shown to converge to the solution of the two-dimensional stochastic Allen-Cahn equation on the torus. Their approach however doesn't imply the tightness for the invariant measure of the model, which is treated in this article. The same result in one space dimension had previously been proven in \cite{MR1317994,Fritz1995}, via a coupling with a simpler model, the voter model. In a subsequent paper \cite{tsatsoulis2016spectral}, Tsatsoulis and Weber show the exponential convergence to equilibrium for the dynamical $\Phi^4_2$ model. 
In \cite{cassandro1995corrections}, the authors show the convergence of the 2d Ising-Kac model on $\Z^2$ to $\Phi^4_2$ by proving the convergence of the discrete Schwinger functions. In particular they were the first to explain, to the best of our knowledge, the small shift of the critical temperature for the Ising-Kac model with the renormalisation constants of the Wick powers. That result (see \cite[Thm~2]{cassandro1995corrections}) is however restricted to temperatures satisfying a condition allowing to use Aizenman's correlation inequalities, which corresponds to large \textit{negative}
values of $A$ in \eqref{e:Phi42measure}.

Our main result resembles the one obtained in \cite{cassandro1995corrections}, with some differences. We will work on a periodic lattice instead of $\Z^2$, which we think of as a discretisation of a 2D torus. This restriction is mainly due to our techniques for bounding the solutions globally in time and a posteriori doesn't appear to be strictly necessary since the limiting dynamic can be defined also on the whole 2D plane (see \cite{MourratWeberGlobal}). Moreover, as our proof exploits the dynamical version of the model and not the correlation inequalities, we do not have the restriction on the temperature present in \cite[Thm~2]{cassandro1995corrections}, so that we
cover arbitrary values $A \in \R$ in \eqref{e:Phi42measure}.

The structure of the present article is as follows: our main result is Theorem~\ref{trm:tightness} showing
tightness of the fluctuations of local averages of the magnetic field in a distributional space. The proof is based on the analysis of the dynamical $\Phi^4_2$ model in \cite[Sec.~3]{tsatsoulis2016spectral} and makes no use of correlation inequalities (not explicitly at least), avoids the restriction (1.8) of \cite{cassandro1995corrections} and exploits the regularisation provided by the time evolution of the Glauber dynamic. As a consequence of Theorem~\ref{trm:tightness}, we obtain in Corollaries~\ref{cor:tightness1} and \ref{cor:tightness2} tightness in $\mathcal{S}'(\T^2)$ for the fluctuation fields \eqref{e:discretefluctuationb}. 
 
In Theorem~\ref{trm:characterisation} we characterise the limit of each subsequence to be an invariant measure 
for the dynamical $\Phi^4_2$ model constructed in \cite{dPD}.
Since it was shown in \cite{dPD} that \eqref{e:Phi42measure} is such a measure and in \cite{tsatsoulis2016spectral} that this
invariant measure is unique, the result follows. For the proof, we make use of the uniform convergence to the invariant measure and the convergence of the Glauber dynamic in \cite{MourratWeber}.

\subsection{Notations}
\label{subsec:notations}
We shall consider spins arranged on a periodic lattice that we will think as embedded into a two-dimensional
torus $\T^2 = [-1,1]^2$. Let $\epsilon = N^{-1}$ and $\Le \eqdef \epsilon \Lambda_N \subset \T^2$ the 
discretisation induced on $\T^2$. 
For $f,g:\Le \to \C$ we define
\[
\Norm{L^p(\Le)}{f}^p \eqdef \sum_{x \in \Le} \epsilon^2 |f(x)|^p\;,\qquad \ang{f,g}_{\Le} \eqdef \sum_{x \in \Le} \epsilon^2 f(x) \overline{ g(x)}\;,
\]
respectively the discrete $L^p$ norm and the scalar product. We will use the discrete convolution
\[
(f \ast g)(x)\eqdef \sum_{y \in \Le}\epsilon^2 f(x-y)g(y)\;,\qquad\text{ for } x \in \Le\;,
\]
and, when there is no possibility of confusion, we will drop the set $\Le$ from the above definitions. We will make an extensive use of the Fourier transform
\[
\widehat{f}(w) \eqdef \ang{f, e_w}_{\Le}\;,\qquad e_w(x) = e^{i \pi w \cdot x}\qquad \text{ for } w \in \Lambda_N\;.
\]
It will sometimes be convenient to also set $e_w = 0$ for $w \in \Z^2 \setminus \Lambda_N$.
With this notation, the Fourier inversion formula reads
\begin{equation}
\label{e:Fourierinversion}
f(x) = \frac{1}{4} \sum_{w \in \Z^2}  \widehat{f}(w) e_w(x) \qquad \text{ for } x \in \Le \:.
\end{equation}
We shall use the same notation $\Ex(f)$ as in \cite{MourratWeber} to denote the extension of $f$ to the continuous torus $\T^2$ via \eqref{e:Fourierinversion} applied to $x \in \T^2$. We recall furthermore the fact that the operator $\Ex$ doesn't commute with the operation of taking the product. (Of course we could have used extensions that do commute, but $\Ex$ behaves nicely with respect to the scale of Besov spaces.) We will measure the regularity of a function $g:\T^2 \to \R$ (or $g:\Le \to \R$) with the Besov norm, defined for $\nu \in \R$, and $p,q \in [1,\infty]$ as
\begin{equ}[e:Besov]
\Norm{\Bb^{\nu}_{p,q}}{g} = \begin{cases}
\left( \sum_{k \geq -1} 2^{\nu k q} \Norm{L^p(\T^2)}{\delta_k g}^q \right)^{\frac{1}{q}} & \text{ if }q < \infty\\
\sup_{k \geq -1} 2^{\nu k} \Norm{L^p(\T^2)}{\delta_k g}& \text{ if }q = \infty
\end{cases}
\end{equ}
(see \eqref{e:PaleyLittlewoodproj} below for the definition of the Paley-Littlewood projection $\delta_k$) and we will denote 
by $\Bb^{\nu}_{p,q}$ the completion of the set of smooth test functions over the torus equipped with the 
corresponding Besov norm. We 
shall denote by $\Cc^{\nu}$ the (separable) Besov space $\Bb^{\nu}_{\infty,\infty}$. In particular, the parameter 
$\nu \in \R$ represents the regularity of a function and the space $\Bb^{\nu}_{p,q}$ contains distributions if $\nu < 0$.
It will be useful to consider, for $g : \Le \to \R$, a discrete version of the Besov norm, that we shall denote 
by $\Norm{\Bb^{\nu}_{p,q}(\Le)}{g}$ (resp. $\Norm{\Cc^{\nu}(\Le)}{g}$),
\[
\Norm{\Bb^{\nu}_{p,q}(\Le^d)}{f} \eqdef \begin{cases}
\left( \sum_{k \geq -1} 2^{\nu k q} \Norm{L^p(\Le^d)}{\delta_k \Ex(f)}^q \right)^{\frac{1}{q}} & \text{ if }q < \infty\\
\sup_{k \geq -1} 2^{\nu k} \Norm{L^p(\Le^d)}{\delta_k \Ex(f)}& \text{ if }q = \infty
\end{cases}
\]
see Section~\ref{sec:discreteBesov} for a more precise description and for some useful properties of this norm used in the article.

\subsection*{Acknowledgements}

{\small
We are grateful to H.~Weber and P.~Tsatsoulis for many discussions on the topic of this article.
MH gratefully acknowledges financial support from the Leverhulme trust
as well as the ERC via the consolidator grant 615897:CRITICAL.
}

\section{Definitions and statements of the theorem}
\label{sec:def-stat}
Assume for the moment that $b=0$, which is also the case studied in \cite{MourratWeber} and consider for $x \in \Lambda_N$
\begin{equation}
\label{e:hgdef}
\hg(x) \eqdef \sum_{z \in \Lambda_N} \kg(x-z) \sigma_z\;,
\end{equation}
where the kernel is the same as in \eqref{e:hamiltonian}.

Following \cite{MR1317994,MourratWeber}, we define the magnetisation fluctuation field over the lattice $\Le$ as $\Xg(z) = \gamma^{-1} \hg(\epsilon^{-1}z)$.
We will consider a dynamic of Glauber type on $\Sigma_N$ in order to gain insight into the 
properties of the fluctuations. In order for this dynamic to converge to a non-trivial limit, we will 
enforce the relation between the scalings $\epsilon$ and $\gamma$ given by \eqref{e:scalarbound}. 

The dynamic can be described informally as follows. Each site $x \in \LN$ is assigned an independent 
exponential clock with rate $1$. When the clock rings, the corresponding spin changes sign with probability
\begin{equation}
\cg(z,\sigma) = \frac{1}{2}\left( 1 - \sigma_z \tanh\left(\beta \hg(z)\right) \right)\;,
\end{equation}
and remains unchanged otherwise.
More formally, the generator of this dynamic is given by
\begin{equation}
\label{e:generatorGlauberch3}
\L_{\gamma} f(\sigma) =  \sum_{z \in \LN} \cg(z,\sigma)\left( f(\sigma^{\{z\}})-f(\sigma) \right)\;,
\end{equation}
for $f: \SN \to \R$, where 
\[
\sigma^{\{z\}}_y = \begin{cases}
-\sigma_z & \text{ if } y = z,\\
\sigma_y & \text{ if } y \neq z.
\end{cases}
\]
The probabilities $\cg(z,\sigma)$ are chosen precisely in such a way that 
$\Pgbeta$ is invariant for this Markov process.
We shall use the notations $\sigma_x(s)$ and $\hg(s,x)$ to refer to the process at (microscopic) space 
$x \in \Lambda_N$ and time $s \in \R_+$. We will use the notation $\P_{\beta,0}^{\gamma}$~(resp.~$\E_{\beta,0}^{\gamma}$) to refer to the probability (resp.~expectation) of the process started 
with an initial condition drawn from $\Pgbeta$.

In order to rewrite the process in macroscopic coordinates, we speed up the generator $\L_{\gamma}$ by a factor $\alpha^{-1}$ and we will abuse the notation writing
\begin{equation}
\label{e:definitionXgGlauber}
\Xg(s,x) \eqdef \delta^{-1}\hg(\alpha^{-1}s,\epsilon^{-1}x)\;,
\end{equation}
in (macroscopic) space $x \in \Le$ and time $s \in \R_+$.
In \cite[Thm~3.2]{MourratWeber} it is proven that, if the parameters $\delta, \alpha, \epsilon$ and the inverse temperature $\beta$ are chosen such that
\begin{equation}
\label{e:scaling}
\delta = \gamma \:, \qquad \alpha = \gamma^2 \:, \qquad \epsilon = \gamma^2 \:, \quad \beta - 1 = \alpha \left( \CG + A\right)\;,
\end{equation}
where $\CG$ is described in \eqref{e:renormconst} below, and if the sequence of initial conditions satisfies $\Xg(0) \to \Xng$ in $\Cc^{-\nu}$, then the law of $\Xg$ on $\Dd(\R_+,\Cc^{-\nu})$, converges in distribution to the solution of the stochastic quantisation equation
\begin{equ}[e:dynamicalPhi42]
\partial_t X = \Delta X - \frac{1}{3}   X^{:3:} + A X + \sqrt{2} \xi\;,\qquad
X(0,\cdot) = X^0 \in \Cc^{-\nu}
\end{equ}
where $\Xg(\cdot,0) \to X_0$ in $\Cc^{-\nu}$ and $\xi$ denotes space-time white noise. The expression $X^{:3:}$ stands for a renormalised power defined as in \cite{dPD}, where the relevant notion of ``solution'' to \eqref{e:dynamicalPhi42} is also given. The solution theory of \eqref{e:dynamicalPhi42} will be briefly summarised in Subsection~\ref{subsec:limitingeq}. The use of the renormalised powers is necessary since the solution belongs to a distributional space.



For $x \in \Le$, let $\Kg(x) = \epsilon^{-2} \kg(\epsilon^{-1} x)$, the macroscopic version of the kernel $\Kg$, and define the discrete Laplacian $\Deltag f = \epsilon^{-2}\gamma^2 (\Kg \ast f - f)$. Under the Glauber dynamic, the process $\Xg$ satisfies on $[0,T] \times \Le$
\begin{equs}
\Xg&(t,x) = \Xg(0,x) \label{e:Xgequation}\\
&+ \int_0^t \Deltag \Xg(s,x) - \frac{1}{3} \left( \Xg^3(s,x) -\CG \Xg(s,x) \right) + A \Xg(s,x) + \Oo(\gamma^2 \Xg^5(s,x))\ ds\\
&+ \Mg(t,x) 
\end{equs}
where $\Mg(t,x)$ is a martingale and $\CG$ is the logarithmically diverging constant
\begin{equation}
\label{e:renormconst}
\CG = \frac{1}{4}\sum_{\omega \in \Lambda_N \setminus \{0\}} \frac{|\hat{\Kg}(\omega)|^2}{\epsilon^{-2}\gamma^2(1 - \hat{\Kg}(\omega))}\:.
\end{equation}

The next theorem is the main result of the paper. Recall the definition of the Besov norm given in Subsection~\ref{subsec:notations}. We think of $\Xg(z)$ as being a random function on $\T^2$, having been extended to $\T^2$ with the $\Ex$ operator.
\begin{theorem}
\label{trm:tightness}
Assume $b=0$. Then for all positive $\nu > 0$ and for all $q > 0$
\[
\limsup_{\gamma \to 0} \Egbeta\left[\Norm{\Cc^{-\nu}}{\Xg}^q \right] < \infty\;.
\]
In particular, the laws of $\Xg$ form a tight set of probability measures on $\Cc^{-\nu}$.
\end{theorem}
From the above theorem it is possible to deduce
\begin{corollary}
\label{cor:tightness1}
Assume $b=0$. Then the law of the field $\left(\gamma^{-1}\sigma_{\lfloor \epsilon^{-1} x\rfloor}\right)_{x \in \T^2}$ is tight in $\mathcal{S}'(\T^2)$ under $\Pgbeta$.
\end{corollary}

\begin{proof}
We will actually prove the tightness in the stronger norm of $H^{-k}(\T^2)$, for $k$ sufficiently big. Let $\varphi \in \mathcal{S}(\T^2)$ and consider
\[
\ang{\gamma^{-1}\sigma_{\lfloor \epsilon^{-1} \cdot\rfloor}, \varphi}_{\T^2} = \sum_{x \in \LN} \epsilon^2\left(\gamma^{-1}\sigma_{x} \right)\bar{\varphi}(\epsilon x)
\]
where $\bar{\varphi}(\epsilon x) = \epsilon^{-2} \int_{|y|_{\infty}\leq 2^{-1}\epsilon} \varphi(\epsilon x+y)\ dy$. Using the differentialbility of $\varphi$, we replace $\bar{\varphi}$ with $\kg \ast \bar{\varphi}$ at the cost of 
\[
\epsilon^2 \gamma^{-2} \sup_{i_1,i_2\in \{1,2\}}\Norm{L^{\infty}(\T^2)}{\partial_{i_1}\partial_{i_2}\varphi}\;,
\]
and this is $\Oo(\gamma^2)$, as $\gamma\to 0$, if $k$ is sufficiently big. Therefore (recall the form of the extension \eqref{e:Fourierinversion} of $\Xg$ to the continuous torus)
\[
\ang{\gamma^{-1}\sigma_{\lfloor \epsilon^{-1} \cdot\rfloor}, \varphi}_{\T^2} = \ang{\Xg, \varphi}_{\T^2} + \Oo(\gamma)
\]
the corollary follows from Theorem~\ref{trm:tightness}.
\end{proof}

As remarked in the proof, the topology with respect to which the convergence in Corollary~\ref{cor:tightness1} is proved is not the optimal norm. Indeed we expect the result to hold also with respect to the norm of $\Cc^{-\nu}$. In the proof of Corollary~\ref{cor:tightness1} we didn't only show the tightness of the sequence of random variable, but we also proved that the limit of $\ang{\gamma^{-1}\sigma_{\lfloor \epsilon^{-1} \cdot\rfloor}, \varphi}_{\T^2}$ coincide with $\lim_{\gamma\to 0}\ang{\Xg, \varphi}_{\T^2}$ for all $\varphi$ sufficiently smooth.

We now show how to extend the previous result to the case $b \neq 0$. It is clear that, by symmetry it is sufficient to assume $b \geq 0$.
In the case of ferromagnetic pair potential $\kg \geq 0$ with positive external magnetisation $b \geq 0$, one has
\begin{equation}
\label{e:correlationineq}
\Egbetab\left[\sigma_x ;  \sigma_y \right] \leq \Egbeta\left[\sigma_x ;  \sigma_y \right]
\end{equation}
where $\Egbetab\left[\sigma_x ;  \sigma_y \right]$ is the covariance between the spins. This follows from the fact that
$
\frac{d}{db}\Egbetab\left[\sigma_x ;  \sigma_y \right] \leq 0
$,
which is an immediate consequence of the GHS inequality (see for instance \cite{Lebowitz1974inequalities} for a proof), valid for $\kg \geq 0$ and $b \geq 0$.

\begin{corollary}
\label{cor:tightness2}
Consider any map $\gamma \mapsto b_\gamma \geq 0$ and denote by $m_{\gamma}(b) = \Egbetab[\sigma_x]$ the mean of the spin $\sigma_x$, which is independent of $x \in \LN$. Then the law of the field 
\begin{equ}[e:scaled]
\tilde X_\gamma(x) = \gamma^{-1}\bigl(\sigma_{\lfloor \epsilon^{-1} x\rfloor} - m_{\gamma}(b_\gamma)\bigr)\;,\qquad x \in \T^2\;,
\end{equ}
is tight in $\mathcal{S}'(\T^2)$ under $\Pgbetabg$.
\end{corollary}

\begin{proof}
Fixing a test function $\phi$ and replacing $\bar{\varphi}$ with $\kg \ast \bar{\varphi}$ as in Corollary~\ref{cor:tightness1}, we have
\[
\ang{\tilde X_\gamma, \varphi} = \sum_{x \in \LN} \epsilon^2\gamma^{-1}\big(\sigma_{x}-m_{\gamma}(b_\gamma) \big)(\kg \ast \bar{\varphi})(\epsilon x) + \Err\;,
\]
where $\Err$ converges to $0$ in probability as $\gamma \to 0$.
Decompose $\varphi = \varphi^+ - \varphi^-$ into its positive and negative part. For each of them, using the correlation inequality \eqref{e:correlationineq}, we have that
\[
\Egbetabg\Bigg|\epsilon^2  \sum_{x \in \LN}  \frac{\sigma_{x}-m_{\gamma}(b_\gamma) }{\gamma}(\kg \ast \overline{\varphi}^{\pm})(\epsilon x) \Bigg|^2 \leq \Egbeta\Bigg|  \epsilon^2\sum_{x \in \LN} (\gamma^{-1}  \sigma_{x})\ (\kg \ast \overline{\varphi}^{\pm})(\epsilon x) \Bigg|^2\;.
\]
Using Theorem~\ref{trm:tightness} we see that, for $\nu \in (0,1)$, this quantity is bounded uniformly
by a fixed multiple of 
$\Norm{\Bb^{\nu}_{1,1}}{\varphi^{\pm}}^2\Egbeta\left[ \Norm{\Cc^{-\nu}}{\Xg}^2 \right]$, up to an error of order $\Oo(\gamma)$.
In order to conclude, we observe that
\[
\Norm{\Bb^{\nu}_{1,1}}{\varphi^{\pm}} \lesssim \Norm{L^1}{\varphi^{\pm}} + \Norm{\text{Lip}}{\varphi^{\pm}}^{\nu}\Norm{L^1}{\varphi^{\pm}}^{1-\nu} \lesssim \Norm{L^1}{\varphi} + \Norm{L^{\infty}}{\nabla \varphi}
\]
where the first inequality is \eqref{e:besovderivboundcontinuous}, generalised to Lipschitz functions.
\end{proof}

The next theorem shows that in the symmetric case $b=0$, the limit of these measures is given by the 
$\Phi^4_2$ measures, as already suggested in \cite{MourratWeber}.

\begin{theorem}
\label{trm:characterisation}
Assume $b=0$. Then any limiting law of the sequence $\{\Xg\}_{\gamma}$ is invariant for the dynamic \ref{e:dynamicalPhi42}, and hence, by \cite{tsatsoulis2016spectral} and \cite[Remark~4.3]{dPD}, coincides with the $\Phi^4(\T^2)$ measure.
\end{theorem}
\begin{proof}
In order to compare the law of a (discrete) random field $\Xg$ with fields on the torus $\T^2$, we will use the extension operator $\Ex$ defined after \eqref{e:Fourierinversion}. For the sake of precision we will explicitly write $\Ex(\Xg(t))$ where the process $\Xg$ has been extended to the whole torus.

We will use the Glauber dynamic and the solution of the stochastic quantisation equation \eqref{e:dynamicalPhi42} introduced in the previous section: the idea is to exploit the exponential convergence to the invariant measure of the solution of the SPDE \eqref{e:dynamicalPhi42} proved in \cite{tsatsoulis2016spectral} and the convergence of the Glauber dynamic of the Kac-Ising model in \cite{MourratWeber}.

By \cite[Thm~3.2]{MourratWeber}, we know that if for $0 < \kappa <\nu$ small enough the sequence of initial conditions $\Ex(\Xng)$  is bounded in $\Cc^{-\nu+\kappa}$ and converges to a limit $\Xn$ in $\Cc^{-\nu}$ as $\gamma \to 0$, one has 
\begin{equ}[e:convGeneral]
\Ex\left(\Xg\right) \stackrel{\mathcal{L}}{\longrightarrow} X \qquad \text{ in }\mathcal{D}\left([0,T];\Cc^{-\nu}\right)\;.
\end{equ}
where $X$ solves \eqref{e:dynamicalPhi42} starting from $\Xn$. In the above equation we took into account the fact that $\Xg$ is defined on the discrete lattice and therefore has to be extended with the operator $\Ex$ to be comparable with $X$.

We first want to show that \eqref{e:convGeneral} holds true when instead of a deterministic sequence $\Ex\Xng \to \Xn$ in $\Cc^{-\nu}$, we have the convergence in law of the initial conditions $\mathcal{L}(\Ex\Xng) \to  \mathcal{L}(\Xn)$ in the topology of $\Cc^{-\nu}$.
In order to do this call $\mathfrak{L}_{\gamma}$ (resp. $\mathfrak{L}_{0}$) the laws at time zero of the processes $\Ex\Xg$ (resp. $X$) and assume that $\mathfrak{L}_{\gamma} \to \mathfrak{L}_{0}$.
Consider then a bounded continuous function $G :\mathcal{D}\left([0,T];\Cc^{-\nu}\right) \to \R$: we want to show that
\[
\lim_{\gamma \to 0} \left| \E\big[ G(\Ex (\Xg)) \big| \Xng \sim \mathfrak{L}_{\gamma} \big] - \E \big[ G(X) \big| \Xn \sim \mathfrak{L}_0  \big] \right| = 0\:.
\]
Conditioning over the initial conditions we can define
\begin{align*}
f_{G}^{\gamma}(\Xng) &:= \E \left[ G(\Ex (\Xg))\Big| \Xg(0) = \Xng \right]\\
f_G(\Xn) &:= \E \left[ G(X)\Big| X(0) = \Xn  \right]\:.
\end{align*}
The result \cite[Thm~3.2]{MourratWeber} implies that $f_{G}^{\gamma}(\Xng) \to f_{G}(\Xn)$ whenever $\Ex\Xng \to \Xn$ in $\Cc^{-\nu}$. Since $\Cc^{-\nu}$ is separable, we can apply the Skorokhod's representation theorem to deduce that there is a probability space $(\tilde{\P},\tilde{\mathcal{F}},\tilde{\Omega})$ where all the processes $\Ex(X^0_{\gamma})$ and $X^0$ can be realised and the sequence $\Ex(X^0_{\gamma})(\tilde{\omega})$ converge to $X^0(\tilde{\omega})$  in $\Cc^{-\nu}$ for $\tilde{\P}$-a.e. $\tilde{\omega} \in \tilde{\Omega}$.

An application of the dominated convergence theorem then shows that, as $\gamma \to 0$
\begin{multline}
\left| \E\big[ G(\Ex (\Xg)) \big| \Xng \sim \mathfrak{L}_{\gamma} \big] - \E \big[ G(X) \big| \Xn \sim \mathfrak{L}_0  \big] \right| \\
\leq \int \left| f_{G}^{\gamma}(\Xng(\tilde{\omega})) - f_{G}(\Xn(\tilde{\omega})) \right| \tilde{\P}(d\tilde{\omega}) \to 0\;,
\end{multline}
so that we can assume \eqref{e:convGeneral} to hold even when the initial datum is convergent in law.

By Theorem~\ref{trm:tightness} we know that, if at time $0$ the configuration $\sigma(0) \in \Sigma_N$ is distributed according to  $\Pgbeta$, then the law of $\Xng(x) = \gamma^{-1}\kg \ast \sigma_{\left\lfloor\epsilon^{-1}x \right\rfloor}(0)$ is tight, and therefore there exists a subsequence $\gamma_{k}$ for $k \geq 0$ and a measure $\mu^*$ on $\Cc^{-\nu}$ such that the law of $\Ex X^0_{\gamma_k}$ converges to $\mu^*$. In the following calculations we will tacitly assume $\gamma \to 0$ along the sequence $\gamma_k$ to avoid the subscript. We will show that, if $\mu$ if the unique invariant measure of \eqref{e:dynamicalPhi42} then $\mu^* = \mu$.

Let $F:\Cc^{-\nu} \to \R$ be a bounded and continuous function, then, by the invariance of the Gibbs measure under the Glauber dynamic, for $t \geq 0$
\[
\Egbeta \left[F(\Ex\Xng)\right] = \E_{\beta,0}^{\gamma} \left[F(\Ex \Xg(t))\right] \:.
\]
Recall that the evaluation map, that associates to a process in $\mathcal{D}\left([0,T];\Cc^{-\nu}\right)$ its value at a given time, is not continuous with respect to the Skorokhod topology, however the integral map
$
G : u \mapsto  \int_0^T F(u(s))\ ds
$
is continuous in its argument in virtue of the the continuity and boundedness of $F$. Hence for any fixed $T$ we have
\[
\Egbeta \left[F(\Ex\Xng)\right] = \E^{\gamma}_{\beta,0}\Bigg[ T^{-1}\int_0^T F\left( \Ex \Xg(s) \right) \ ds \Bigg]
\]
and
\[
\lim_{\gamma \to 0}\left| \E^{\gamma}_{\beta,0}\Bigg[ \int_0^T F\left( \Ex \Xg(s) \right) \ ds \Bigg] -  \E\Bigg[ \int_0^T F\left( X(s) \right) \ ds\Bigg| X(0) \sim \mu^*\Bigg] \right| = 0\;.
\]
By the uniform convergence to equilibrium of the stochastic quantisation equation \cite[Cor.~6.6]{tsatsoulis2016spectral} there exist constants $c,C > 0$
\[
\left| \E[F\big( X(s)\big)\big| X(0) \sim \mu^*] - \mu[F]\right| \leq C \norm{\infty}{F} e^{-c s}\:.
\]
From the above inequality it follows that
\[
 \left| T^{-1}\int_0^T \E[F\big( X(s)\big)\big| X(0) \sim \mu^*] -\mu[F]   ds\right| \lesssim T^{-1} \norm{\infty}{F}
\]
and letting $T$ be large enough the last difference can be made arbitrarily small.
From the above estimates we can see that, for arbitrary $T>0$,
\[
\limsup_{\gamma \to 0}|\Egbeta \left[F(\Ex\Xng)\right] -  \mu [F] |\leq C \norm{\infty}{F} T^{-1}
\]
and the result follows.
\end{proof}

\begin{remark}
For $b_\gamma = b$ constant, we actually expect the limiting points to vanish under the scaling \eqref{e:scaled}.
On the other hand, for $b_\gamma = b\gamma$, one can follow an argument
virtually identical to the one given in this article to show that the limit is given by the law
of the $\Phi^4_2$ measure with external magnetic field $b$.
\end{remark}

\subsection{Solution of the limiting equation}
\label{subsec:limitingeq}
Before the proof of the main theorem, let us briefly explain the construction of the solution in \cite{dPD} to the following SPDE
\[
d X = \Big( \Delta X + \sum_{j=1}^n a_{2j-1} X^{2j-1}\Big)dt + \sqrt{2}dW\;.
\]
As in \eqref{e:dynamicalPhi42}, the powers in the above SPDE have to be renormalised in order to find a nontrivial solution. The precise way the process is renormalised follows \cite{MourratWeber,ShenWeber}.
Consider at first $Z(t)$ the solution of the stochastic heat equation
\[
d Z = \Delta Z dt+ \sqrt{2}dW\,\qquad Z(\cdot,0) = 0\:,
\]
and therefore in two dimension $Z$ belongs to $\Cc([0,T];\Cc^{-\nu})$ a.s. for any $\nu > 0$. Consider the Galerkin approximation
\begin{equation}
\label{e:sheapprox}
d Z_{\epsilon} = \Delta Z_{\epsilon} dt+ \sqrt{2}dW_{\epsilon}\,\qquad Z_{\epsilon}(\cdot,0) = 0\:.
\end{equation}
From the above SPDE we see that $Z_{\epsilon}$ has a representation in terms of the stochastic convolution
\[
Z_{\epsilon}(t,x) = \frac{\sqrt{2}}{4}\sum_{\omega \in \LN}\int_0^t e^{-(t-s)\pi^2|\omega|^2}d\hat{W}(s,\omega)\:.
\]
Define the renormalisation constant
\[
\mathfrak{c}_{\epsilon}(t) \eqdef \E\left[Z_{\epsilon}^2(t,0)\right] = \frac{1}{2}\sum_{\omega \in \LN \setminus \{0\}} \int_0^t e^{-2(t-s)\pi^2|\omega|^2}ds =  \sum_{\omega \in \LN\setminus \{0\}} \frac{1 - e^{-2t\pi^2|\omega|^2}}{4\pi^2|\omega|^2}
\]
and its time independent version
\[
\mathfrak{c}_{\epsilon} := \lim_{t \to \infty}\E\left[Z_{\epsilon}^2(t,0) - \frac{t}{2}\right] = \sum_{\omega \in \LN\setminus \{0\}} \frac{1 }{4\pi^2|\omega|^2}\:.
\]
In order to renormalise the process $Z_{\epsilon}(t)$ at finite time, it is more convenient to use $\mathfrak{c}_{\epsilon}(t) \eqdef \E[Z_{\epsilon}^2(0,t)]$. We therefore define the renormalised powers of the process $\Z_{\epsilon}$ as
\[
Z_{\epsilon}^{:n:}(t) \eqdef H_n(Z_{\epsilon}(t),\mathfrak{c}_{\epsilon}(t))\:.
\]
By \cite[Lem.~3.2]{dPD}, the process $Z_{\epsilon}^{:n:}(t)$ is Cauchy in $L^p\left(\Cc([0,T];\Cc^{-\nu}),\P\right)$ for every $p \geq 1$ and we will be referring to its limit as $Z^{:n:}$.
To be precise, the result in \cite[Lem.~3.2]{dPD} is proven for a fixed time, but the extension to the whole process is immediate.

By the variation of constants formula, the solution to \eqref{e:sheapprox} started from the initial condition $\Xn_{\epsilon} \eqdef \Pi_{\epsilon}\Xn$ is given by $\tilde{Z}_{\epsilon} (t)= e^{\Delta t} \Xn_{\epsilon} + Z_{\epsilon}(t)$.
To extend the definition of the renormalised powers to the process $\tilde{Z}_{\epsilon}(t)$
one uses the following property of the Hermite polynomial
\[
H_n(a+b,c) = \sum_{j=0}^n \binom{n}{j} b^{n-j} H_j(a,c)
\]
and let 
\[
\tilde{Z}_{\epsilon}^{:n:}(t) = \sum_{j=0}^n \binom{n}{j} \left( e^{\Delta t} \Xn_{\epsilon} \right)^{n-j}  Z_{\epsilon}^{:j:}(t)\:.
\]
The above random variable is well defined because $e^{\Delta t} \Xn_{\epsilon}$ is a smooth function and the product with $Z_{\epsilon}^{:j:}(t)$ is in $\Cc^{-\nu}$ for any $t>0$ (see for instance \cite[Cor.~3.2]{MourratWeberGlobal} or Theorems~2.82 and 2.85 in \cite{bahouri2011fourier} for a proof).

We then set $X_{\epsilon}(t)$ the Galerkin approximation of $X(t)$ solving
\begin{equation}
\label{e:firstpolydiscrete}
\begin{cases}
d X_{\epsilon}(t) &= \left( \Delta X_{\epsilon}(t) + \sum_{j=1}^n a_{2j-1} H_{n}(X_{\epsilon}(t),\mathfrak{c}_{\epsilon})\right)dt + \sqrt{2}dW_{\epsilon}(t)\\
X_{\epsilon}(0) &= \Pi_{\epsilon}\Xn \in \Cc^{-\nu}
\end{cases}
\end{equation}
Since $H_n$ is a polynomial in both variables, it is possible to replace $\mathfrak{c}_{\epsilon}$ in the above formula with $\mathfrak{c}_{\epsilon}(t)$ provided one compensates it in the coefficient of the polynomial.
\[
\sum_{j=1}^n a_{2j-1} H_{n}(X_{\epsilon}(t),\mathfrak{c}_{\epsilon}) = \sum_{j=1}^n a_{2j-1}(t,\epsilon) H_{n}(X_{\epsilon}(t),\mathfrak{c}_{\epsilon}(t))
\]
with new coefficients $a_{2j-1}(t)$ depending polynomially only on the old coefficients and on the difference $\mathfrak{c}_{\epsilon}-\mathfrak{c}_{\epsilon}(t)$. From the definitions of $\mathfrak{c}_{\epsilon}$ and $\mathfrak{c}_{\epsilon}(t)$ one can see that their difference is diverging logarithmically as $t \to 0$ and therefore each power of $a_{2j-1}(t)$ is integrable in $[0,T]$. Hence we can rewrite \eqref{e:firstpolydiscrete} as
\begin{equation}
\begin{cases}
d X_{\epsilon}(t) &= \left( \Delta X_{\epsilon}(t) + \sum_{j=1}^n a_{2j-1}(t,\epsilon) H_{n}(X_{\epsilon}(t),\mathfrak{c}_{\epsilon}(t))\right)dt + \sqrt{2}dW_{\epsilon}(t)\\
X_{\epsilon}(0) &= \Pi_{\epsilon}\Xn \in \Cc^{-\nu}
\end{cases}
\end{equation}
We decompose $X_{\epsilon}(t)=\tilde{Z}_{\epsilon}(t)+V_{\epsilon}(t)$ where, for a.e. realisation of $Z_{\epsilon}$, the process $V_{\epsilon}$ solves the PDE
\begin{equation}
\label{e:secondpolydiscrete}
\begin{cases}
\partial_t V_{\epsilon}(t) &=   \Delta V_{\epsilon}(t)  + \sum_{j=1}^n a_{2j-1}(t,\epsilon) H_{n}\left(\tilde{Z}_{\epsilon}(t)+V_{\epsilon}(t),\mathfrak{c}_{\epsilon}(t)\right) \\
V_{\epsilon}(0) &=0
\end{cases}
\end{equation}
where
\[
H_n\left(\tilde{Z}_{\epsilon}(t)+V_{\epsilon}(t),\mathfrak{c}_{\epsilon}(t)\right) = \sum_{j=0}^n \binom{n}{j} V^{n-j}_{\epsilon}(t)  \tilde{Z}^{:j:}_{\epsilon}(t)\:.
\]
The last product is again well-posed thanks to the fact that $V_{\epsilon}(t) \in \Cc^{2-\nu-\kappa}$ for any $\kappa>0$, from the regularizing properties of the parabolic equation \eqref{e:secondpolydiscrete}.

As proven in \cite{dPD}, the processes $V_{\epsilon}$ converge in $\Cc([0,T];\Cc^{2-\nu})$ to the solution of
\begin{equation}
\label{e:thirdpolydiscrete}
\begin{cases}
d V &=   \Delta V(t)  + \sum_{j=1}^n a_{2j-1}(t) \sum_{j=0}^n \binom{n}{j} V^{n-j}(t) \tilde{Z}^{:j:}(t)\\
V(0) &=0
\end{cases}
\end{equation}
where $\tilde{Z}^{:j:}(t) \eqdef \lim_{\epsilon \to 0} P_t\Xn + Z^{:j:}_{\epsilon}(t)$. For all $\kappa>0$, the solution of the above PDE is unique and  $V \in \Cc([0,T];\Cc^{2-\nu-\kappa})$ and only depends on the realisation of the process $Z$ via the tuple $(\tilde{Z},\dots,\tilde{Z}^{:2n-1:}) \in L^{\infty}([0,T];\Cc^{-\nu})^n$. We summarise it with the following proposition, essentially proven in \cite{dPD}.
\begin{theorem}
For all $\nu > 0$ and $T>0$ there exists a locally Lipschitz continuous function 
\[
\mathcal{S}_T :L^{\infty}([0,T];\Cc^{-\nu})^n \to \Cc([0,T];\Cc^{2-\nu-\kappa})
\]
 that associates to $(\tilde{Z},\tilde{Z}^{:2:}\dots,\tilde{Z}^{:2n-1:})$ the solution of \eqref{e:thirdpolydiscrete}.
\end{theorem}

Using the definitions above, we now outline the skeleton of the proof in \cite{MourratWeber}. First of all we want to remark that we made the decision of absorbing the initial conditions in the process $\tilde{Z}_{\epsilon}$, instead we could have started \eqref{e:secondpolydiscrete} from $\Pi_\epsilon \Xng$ and defined a similar solution map $\mathcal{S}_{T}^{X^0}$. Consider $\Zg$ the solution to the linearised part of \eqref{e:Xgequation} satisfying
\begin{equation}
\label{e:Zgequation}
\Zg(t,x) = \int_0^t\Deltag\Zg(s,x) ds + \Mg(t,x)\;,
\end{equation}
which is an approximation to the stochastic heat equation. 
In \cite{dPD} and \cite{MourratWeber} the authors provided a useful definition of $\Zg^{:n:}$ the renormalized powers of $\Zg$ that we will not introduce here. In this article will only use the fact that for any $T>0$, $q > 0$, $\nu > 0$, $j \geq 0$ and $\lambda>0$,
\begin{equation}
\label{e:Hermitebound}
\limsup_{\gamma \to 0}\E \left[ \sup_{s \in [0,T]}s^{\lambda}\Norm{\Cc^{-\nu}}{H_j\left(\Zg(s,\cdot),\CG\right)}^q \right] < \infty\:,
\end{equation}
which follows from Propositions~5.3 and 5.4 and \cite[Eq.~3.15]{MourratWeber}.

In \cite[Sec.~6]{MourratWeber} it is proven that processes $\left(Z_\gamma,Z_\gamma^{:2:}, \dots, Z_\gamma^{:2n-1:} \right)$ jointly converge in law to $(Z,Z^{:2:},\dots,Z^{:2n-1:})$. Using the decomposition $\Xg = \Zg + \Vg$, it is possible to see that $\Vg$ satisfies an equation similar to \eqref{e:secondpolydiscrete},  with initial condition $\Xng$ and
\[
\Norm{L^{\infty}([0,T];\Cc^{-\nu})}{\Vg - \mathcal{S}_T^{\Xng}(Z_\gamma,Z_\gamma^{:2:}, \dots, Z_\gamma^{:2n-1:} )} \to 0\;.
\]
Therefore, by the continuity of $\mathcal{S}_T$, we have that, as $\gamma \to 0$,
\[
\Xg \sim \Zg + S_T^{\Xng}(Z_\gamma,Z_\gamma^{:2:}, \dots, Z_\gamma^{:2n-1:} )\stackrel{\mathcal{L}}{\longrightarrow}  Z + S_T^{\Xng}(Z,Z^{:2:}, \dots, Z^{:2n-1:} ) = X\;,
\]
as required.

\section{Proof of Theorem~\ref{trm:tightness}}
We are now going to prove the statements used in Section~\ref{sec:def-stat} and in particular Theorem~\ref{trm:tightness}.
We first obtain a very suboptimal bound on $\Xg$ which can be used as a starting point for the derivation of sharper bounds.
\begin{proposition}
\label{prop:nonoptimalbound}
Let $p \geq 2$ an even integer, and $\lambda \in [0,1]$ then there exists $C(p,\lambda) > 0$ such that
\[
\E\big[\Norm{L^p(\Le)}{\Xg(t,\cdot)}^p\big] \leq  C \left(\E\big[\Norm{L^p(\Le)}{\Xg(0,\cdot)}^p\big]^{1-\lambda} t^{-\frac{p}{2}\lambda}  \right)\vee \gamma^{-\frac{p}{2}}\;.
\]
In particular, if we start the process from the invariant measure, we obtain that there exists $C = C(p)>0$ such that for all $t \geq 0$
\begin{equation}
\label{e:nonoptimalbound}
\Egbeta\big[\Norm{L^p(\Le)}{\Xg}^p\big] = \E_{\beta,0}^{\gamma}\big[\Norm{L^p(\Le)}{\Xg(t,\cdot)}^p\big] \leq C(p) \gamma^{-\frac{p}{2}}
\end{equation}
\end{proposition}

\begin{proof}
Recall the action of the generator of the Glauber dynamic \eqref{e:generatorGlauberch3}:
\begin{multline*}
\L_{\gamma} \hg^p(t,x) = \sum_{z \in \Lambda_N} \cg(z,\sigma(t)) \Big(\left( \hg(t,x)-2 \sigma_z(t)\kg(z-x) \right)^p - \hg^p(t,x) \Big)\\
\leq p\Big(- \hg + \kg \ast \tanh(\beta \hg)\Big)(t,x) \hg^{p-1}(t,x) + C_p \Big(|\hg(t,x)| + \gamma^2\Big)^{p-2} \gamma^2\:.
\end{multline*}
We can take the average over $x \in \Lambda_N$ to obtain
\begin{multline*}
\L_{\gamma} \Norm{L^p(\LN)}{\hg(t)}^p \leq p \ang{\hg^{p-1}(t), \kg \ast \tanh(\beta \hg(t))}_{ \Lambda_N } - p\Norm{L^1(\Lambda_N)}{\hg^{p}(t)} \\
+ C_p \gamma^2 \Norm{L^1(\Lambda_N)}{\hg^{p-2}(t)} + C_p  \gamma^{2p-2}\;.
\end{multline*}
We use the fact that $p$ is even and the hyperbolic tangent is monotone to bound
\[
\ang{\hg^{p-1}(t), \kg \ast \tanh(\beta \hg(t))}_{\Lambda_N} \leq \ang{\hg^{p-1}(t),\tanh(\beta \hg(t))}_{\Lambda_N}\:.
\]
Moreover, it is easy to see that there exists a constant $c_0>0$ such that
\[
\frac{\tanh(\beta h)}{h} \leq \beta - c_0 h^2 \text{ for } h \in [1,1]\:.
\]
Since $|\hg(t,x)| \leq 1$ and $\beta=1+ \gamma^2 (\CGG+A)$, we can bound $\L_{\gamma} \Norm{L^1(\Lambda_N)}{\hg^p(t)}$ with
\begin{multline*}
p[\beta-1]\Norm{L^1(\Lambda_N)}{\hg^p(t)} - c_0 p	 \Norm{L^1(\Lambda_N)}{\hg^{p+2}(t)} + C_p \gamma^2 \Norm{L^1(\Lambda_N)}{\hg^{p-2}(t)} +C_p \gamma^{2p-2}\\
\leq C ( \gamma^2 \CGG )^{\frac{p+2}{2}} - \frac{c_0}{2} p \Norm{L^{p}(\Lambda_N)}{\hg(t)}^{p+2} + C\gamma^{\frac{p+2}{2}} +C_p \gamma^{2p-2}\\ \leq  - \frac{c_0}{2} p \Norm{L^{p}(\Lambda_N)}{\hg(t)}^{p+2} + C \gamma^{\frac{p+2}{2}}\;,
\end{multline*}
where we used the fact that $|A|\leq \CG$ for $\gamma$ small enough and the generalised Young inequality in the last line. Therefore, taking the expectation
\begin{multline*}
\E\big[\Norm{L^p(\Le)}{\Xg(t)}^p\big] = \E\big[\Norm{L^p(\Le)}{\Xg(0)}^p\big] + \int_0^t \E\big[\L_{\gamma}\Norm{L^p(\Le)}{\Xg(s)}^p\big]\ ds\\
\leq \E\big[\Norm{L^p(\Le)}{\Xg(0)}^p\big]  - \frac{c_0}{2} p \gamma^2 \int_0^t \E\big[\Norm{L^p(\Le)}{\Xg(s)}^p\big]^{\frac{p+2}{p}}\ ds + C \gamma^{\frac{2-p}{2}}\;.
\end{multline*}
From the comparison test in Lemma~\ref{lemma:comparison} we have that 
\[
\E\big[\Norm{L^p(\Le)}{\Xg(t)}^p\big] \lesssim  \frac{\E\big[\Norm{L^p}{\Xg(0)}^p\big] }{\left( 1 + c_p t\E\big[\Norm{L^p(\Le)}{\Xg(0)}^p\big]^{\frac{2}{p}}\right)^{\frac{p}{2}}} \vee \gamma^{-\frac{p}{2}}
\]
and the result follows.
\end{proof}

\begin{remark}
Despite its simplicity, Proposition~\ref{prop:nonoptimalbound} has the advantage of making the
proof of \cite[Thm~6.1]{MourratWeber} simpler, avoiding the need for the stopping time $\taun$ and providing sufficient control over \cite[Eq.~6.7]{MourratWeber}.
\end{remark}
\begin{proposition}
\label{prop:Vgcomparison}
Recall the definitions given in Section~\ref{sec:def-stat} of the processes $\Xg$, $\Zg$ and $\Vg:= \Xg - \Zg$. We want to remark that $\Vg$ is not similar to of $V_{\epsilon}$, because of the initial condition (see \eqref{e:thirdpolydiscrete}). Let $p \geq 2$ an even integer. Then there exist $\nu_0 > 0$, $\lambda_{j,i} > 0$ for $i=1,2$ and $j=0,1,2$ such that for all $0< \nu < \nu_0$ and $0 \leq s  \leq t \leq T$
\begin{multline}
\label{e:Vgcomparison}
\Norm{L^p(\Le)}{\Vg(t,\cdot)}^p - \Norm{L^p(\Le)}{\Vg(s,\cdot)}^p \\
+ C_1 \int_s^t \Norm{L^p(\Le)}{\Vg(r,\cdot)}^{p+2} dr + C_1 \int_s^t \ang{\Vg^{p-1}(r),(-\Deltag)\Vg(r)}_{\Le} dr  \\
\leq  C_2 \int_s^t\sum_{j=0}^3 \sum_{i=1,2}\Norm{\Cc^{-\nu}(\Le)}{H_j(Z_\gamma(r,\cdot),\CG)}^{\lambda_{j,i}}\ dr + \int_s^t\Err(r)\ dr
\end{multline}
where, for every $q>0$
\begin{equation}
\label{e:Vgbounderror}
\sup_{0 \leq r \leq T}\E_{\beta,0}^{\gamma}\left[\Err^q(r) \right]^{\frac{1}{q}} \lesssim C_3(p,q,T) \gamma^{\frac{p-2}{6}-2\nu \frac{(p-2)}{3}}\;.
\end{equation}
\end{proposition}

\begin{proof}
The proof follows the argument in \cite[Prop.~3.7]{tsatsoulis2016spectral}, with the important difference that in our case all the operators are discrete operators. Without loss of generality, we will prove \eqref{e:Vgcomparison} starting at time $s=0$ from $\Vng = \Xng$.

In the following calculations, since there is no possibility of confusion, we will use $L^p$ instead of $L^p(\Le)$, and $\ang{\cdot,\cdot}$ instead of $\ang{\cdot,\cdot}_{\Le}$.
From \eqref{e:Xgequation} and \eqref{e:Zgequation} we see that $\Vg(t,x)$ satisfies, for $x \in \Le, t \geq 0$
\begin{align*}
&\Vg(t,x) \\
&\quad= \Vng(x) + \int_0^t \Deltag \Vg(s,x) ds + \int_0^t \gamma^{-2}\Kg \ast \left( \gamma^{-1}\tanh(\beta \gamma \Xg(s,x)) - \Xg(s,x) \right) ds
\end{align*}
and in particular $\Vg(t,x)$ is continuous and weakly differentiable in time, for all $\gamma>0$. Recall that $\beta = 1 + \gamma^2 (\CGG + A)$ and expand the hyperbolic tangent up to third order
\begin{multline*}
\tanh(\beta \gamma \Xg(s))= \gamma \Xg(s) +  \gamma^3 (\CGG + A) \Xg(s) - \frac{\gamma^3}{3} \Xg^3(s) \\
+ \gamma^3(\beta - 1)\Oo \left( \Xg^3(s) \right)+ \Oo\left(\gamma^5 \Xg^5(s) \right).
\end{multline*}
With the above formula the derivative of the discrete $L^p$ norm of $\Vg$ is calculated
\begin{equation}
\label{e:Vpnormineq1}
\Norm{L^p}{\Vg(t)}^p = \Norm{L^p}{\Vng}^p + p \int_0^t \ang{\Vg^{p-1} , \Deltag \Vg }(s)\ ds + \frac{1}{3}  D(s) +  B(s)\ ds
\end{equation}
where 
\[
D(s) = - \ang{ \Kg \ast \Vg^{p-1}(s) , \Xg^3(s) - 3(\CGG + A )\Xg(s)}
\]
and $B(s)$ is produced by the remainder of the Taylor expansion of the hyperbolic tangent
\begin{equation}
\label{e:Vpnormineq1B}
B(s) \leq C \gamma^2   \ang{|\Vg^{p-1}|(s),\CGG |\Xg|^3(s) + |\Xg|^5(s)  }\:.
\end{equation}
where we used the fact that $|A| \leq \CGG$ for $\gamma$ small enough.We will first replace $D(s)$ with
\begin{align}
\label{e:A_1inequality}
D_1(s) &:= -\ang{\Vg^{p-1}(s)  ,\Xg^3(s)  - 3\CGG \Xg(s) } + 3A \ang{\Vg^{p-1}(s)  , \Xg(s) }\\
\nonumber
&\leq -\Norm{L^1}{\Vg^{p+2}(s)} + 3|\ang{\Vg^{p+1}(s),\Zg(s)}|+ 3|\ang{\Vg^{p}(s),H_2(\Zg(s),\CG)}| \\
\nonumber
&+ |\ang{\Vg^{p-1}(s),H_3(\Zg(s),\CG)}| + 3A \Norm{L^1}{\Vg^{p}(s)} + 3A|\ang{\Vg^{p-1}(s),\Zg(s)}|
\end{align}
Let 
\[
L_s \eqdef \Norm{L^1}{\Vg^{p+2}(s)}\:,\qquad K_s \eqdef \ang{ \Vg^{p-1}(s) , \Deltag \Vg(s) }\:.
\]
Those terms are the good terms of \eqref{e:Vpnormineq1}, and the idea is now to bound all the other errors $|D(s)-D_1(s)|$ with expression containing $L_s$ and $K_s$. In the following calculations we assume $\gamma$ to be small enough such that $|A| \leq \CGG$. The cost of replacing $D(s)$ with $D_1(s)$ is given by
\begin{align*}
|D(s) - D_1(s)| &\leq \sum_{x,y \in \Le} \epsilon^4 \Kg(x-y) \left| \Vg^{p-1}(s,y)- \Vg^{p-1}(s,x) \right| \\
&\quad \times \big|\big(\Xg^3(s,y)  -\Xg^3(s,x) \big)  - 3(\CGG + A) \left( \Xg(s,y) - \Xg(s,x)\right)\big|\\
&\leq 3 \sum_{x,y \in \Le} \epsilon^4 \Kg(x-y) \left| \Vg^{p-1}(s,y)- \Vg^{p-1}(s,x) \right| \\
&\quad\times \big(\big| \Vg(s,y)- \Vg(s,x) \big| + \left| \Zg(s,y)- \Zg(s,x) \right| \big)\left( 2 \CG   +  \Xg^2(s,x)\right)\:.
\end{align*}
Denote with
\begin{align*}
D_2  = 3 \sum_{x,y \in \Le} \epsilon^4 \Kg(x-y)& \left| \Vg^{p-1}(s,y)- \Vg^{p-1}(s,x) \right|\\
&\qquad\times\left| \Vg(s,y)- \Vg(s,x) \right|\left( 2 \CG   +  \Xg^2(s,x)\right)\\
D_3  = 3 \sum_{x,y \in \Le} \epsilon^4 \Kg(x-y) &\left| \Vg^{p-1}(s,y)- \Vg^{p-1}(s,x) \right|\\
&\qquad\times\left| \Zg(s,y)- \Zg(s,x) \right|\left( 2 \CG   +  \Xg^2(s,x)\right)\;.
\end{align*}
We will now bound $D_3$ with a small multiple of $L_s$ and $K_s$ plus an error in \eqref{e:Vgbounderror}, the term $D_2$ can be bounded in a similar way.\\
By $a^n - b^n = (a-b)(a^{n-1} + \dots + b^{n-1})$ and the generalized Young inequality
\begin{multline*}
a^{p-1} - b^{p-1} = (a-b)(a^{p-2} + \dots + b^{p-2})\\
\leq |a^{p-1} - b^{p-1}|\frac{|a-b|}{2 \lambda} + (|a|^{p-1} + |b|^{p-1}) \lambda 2^{p-2}
\end{multline*}
Therefore, applying the previous inequality to each summands of $D_3$ and choosing $\lambda = c_1^{-1}(\gamma^{-1} \epsilon)^2\left| \Zg(s,y)- \Zg(s,x) \right|\left( 2 \CG   +  \Xg^2(s,x)\right)$ we have that
\begin{multline}
\label{e:kgsub}
D_3 \leq c_1 K_s + C c_1^{-1}(\epsilon^2 \gamma^{-2})\sum_{x \in \Le} \epsilon^2 |\Vg^{p-1}(s,x)|| \Zg(s,x)|\left(  \CG   +  \Xg^2(s,x)\right)\\
\leq c_1 K_s + c_1 L_s +C c_1^{-1} \Norm{L^{\frac{p-2}{3}}}{\epsilon^2 \gamma^{-2}| \Zg( s)|\left(  \CG   +  \Xg^2( s)\right)}^{(p-2)/3}
\end{multline}
where $c_1>0$ can be chosen to be for instance $c_1 = 1/8$. The last term will be part of the error \eqref{e:Vgbounderror}. Recall that $\epsilon = \gamma^2$ and the last term of \eqref{e:kgsub} is bounded in expectation using Proposition~\ref{prop:nonoptimalbound}, Lemma~\ref{lemma:Lpdiscretebound} and \eqref{e:Hermitebound}
\begin{multline}
\E_{\beta,0}^{\gamma}\left[\Norm{L^{\frac{p-2}{3}}}{\epsilon^2 \gamma^{-2}| \Zg( s)|\left(  \CG   +  \Xg^2( s)\right)}^{(p-2)/3}\right] \leq \\
\E_{\beta,0}^{\gamma} \left[ \Norm{\Cc^{-\nu}(\T^2)}{ \Zg( s) }^{2(p-2)/3} \right]^{1/2} \left( (\gamma^2 \CG)^{\frac{2(p-2)}{3}}  + \E_{\beta,0}^{\gamma}\left[\Norm{L^{2(p-2)/3}}{   \gamma \Xg( s) }^{2(p-2)/3} \right]\right)^{1/2}\\
\leq C(T) \gamma^{\frac{p-2}{6}-2\nu \frac{(p-2)}{3}}
\label{e:expectationerrorbound}
\end{multline}
which is negligible if $\nu$ is small enough and $p > 2$. It is immediate to generalize \eqref{e:expectationerrorbound} to any power, as in \eqref{e:Vgbounderror}.

We will then bound the term $B(t)$ in \eqref{e:Vpnormineq1B} with Proposition~\ref{prop:nonoptimalbound}. Using Young's inequality we have that
\begin{align*}
B(s) &\leq C \gamma^2   \ang{|\Vg^{p-1}|(s),\CGG + \left|\Zg^2(s)+2\Vg(s)\Zg(s) + \Vg^2(s)\right||\Xg|^3(s)} \\
&\leq \frac{1}{24}  \Norm{L^{p+2}}{\Vg(s)}^{p+2} + C \CGG^{\frac{p+2}{3}} \bigNorm{L^1}{(\gamma^{\frac{2}{3}}\Xg(s))^{p+2} }  \\
&+ C  \bigNorm{L^1}{\Zg^{\frac{2p+4}{3}}(\gamma^{2}\Xg^3(s))^{\frac{p+2}{3}} } +  \bigNorm{L^1}{\Zg^{\frac{p+2}{2}}(\gamma^{2}\Xg^3(s))^{\frac{p+2}{2}} }+  \Norm{L^1}{(\gamma^{2}\Xg^3(s))^{p+2} } \;.
\end{align*}
The constant $1/24$ has been arbitrarily chosen in order to control $B(s)$ with a small multiple of $L_s$ plus a quantity that will be part of the error in \eqref{e:Vgbounderror} and can be bounded in expectation, as we did in \eqref{e:expectationerrorbound}, by $C(T) \gamma^{\frac{p+2}{6} - 2\nu \frac{2p+4}{3}}$, which is negligible for $\nu$ small enough.

We are now in the setting of \cite[Eq.~3.13]{tsatsoulis2016spectral}, namely the discrete process $\Vg$ satisfies 
\begin{align}
\label{e:scalarbound0}
&\Norm{L^p}{\Vg(t)}^p  - p \int_0^t \frac{5}{6} K_s + \frac{5}{24} L_s \ ds \\
\nonumber
&\qquad\leq  \Norm{L^p}{\Vng}^p + \frac{p}{3} \int_0^t \sum_{j=0}^2 \binom{3}{j} \ang{\Vg^{p-1+j}(s) , H_{3-j}(\Zg(s),\CG)}\ ds\\
\nonumber
&\qquad+  A \int_0^t | \ang{\Vg^{p-1}(s) , (\Vg(s) + \Zg(s))}|\ ds + \int_0^t \Err(s) ds\:,
\end{align}
where $\E_{\beta,0}^{\gamma}[|\Err(s)|^q]^{\frac{1}{q}} \leq C(T,p,q) \gamma^{\frac{p-2}{6}-2\nu \frac{(p-2)}{3}}$ for any positive $q$. We will now show that, for $\nu$ small enough and $j = 0,1,2$, there exist $\lambda_{j,1},\lambda_{j,1} > 0$
\begin{multline}
\label{e:scalarbound}
\ang{\Vg^{p-1+j} , H_{3-j}(\Zg(s),\CG)} \\
\lesssim \left( L_s^{\frac{p-1+j}{p+2} - \nu \frac{p }{p+2}} K^{\nu}_s + L_s^{\frac{p-1+j}{p+2}} \right) \Norm{\Cc^{-\nu}(\Le)}{H_{3-j}(\Zg(s),\CG)}\\
\leq \frac{1}{7} K_s + \frac{1}{30} L_s + C \sum_{i=1,2} \Norm{\Cc^{-\nu}(\Le)}{H_{3-j}(\Zg(s),\CG)}^{\lambda_i}
\end{multline}
where the last line follows from the Young inequality for $\nu$ sufficiently small. In a similar way
\begin{equation}
\label{e:scalarbound2}
A| \ang{\Vg^{p-1}(s) , (\Vg(s) + \Zg(s))}| \leq  \frac{1}{7} K_s + \frac{1}{30} L_s + C(A)\left(1 + \Norm{\Cc^{-\nu}(\Le)}{\Zg(s)}^{\lambda_i}\right) 
\end{equation}
Recall that all the norms appearing the proof so far are norms on the discrete lattice. The same proof of \cite[Prop.~3.7]{tsatsoulis2016spectral} can be used to prove \eqref{e:scalarbound} and \eqref{e:scalarbound2}, provided the same inequalities hold in the discrete setting.\\
We are going to prove \eqref{e:scalarbound}, \eqref{e:scalarbound2} being essentially the same. Using the duality for discrete Besov spaces proved in Proposition~\ref{prop:discretedualityBesov}
\[
\ang{\Vg^{p-1+j}( s), H_{3-j}(\Zg(s),\CG)}_{\Le} \leq \Norm{\Bb_{1,1}^{\nu}(\Le)}{\Vg^{p-1+j}( s)} \Norm{\Cc^{-\nu}(\Le)}{H_{3-j}(\Zg(s),\CG)}.
\]
We then control $\bigNorm{\Bb_{1,1}^{\nu}(\Le)}{\Vg^{p-1+j}( s)}$ with Lemma~\ref{lemma:besovregularity}. From \eqref{e:besovderivbound} applied to $f(x) = \Vg^{p-1+j}(s,x)$
\[
\Norm{\Bb_{1,1}^{\nu}(\Le)}{f} \lesssim \Norm{L^1(\Le)}{f}^{1-2\nu}  \left(\sum_{x,y \in \Le} \epsilon^4 \Kg(x-y)\epsilon^{-1}\gamma|f(x)-f(y)| \right)^{2\nu}+  \Norm{L^1(\Le)}{f}\:.
\]
We will now estimate the term inside the brackets. For $p$ even and $j \in \N$, we have
\[
|a^{p-1+j}-b^{p-1+j}|^{\frac{p-1 }{p-1+j}} \leq |a^{p-1}-b^{p-1}|
\]
the above equation follows easily from the Minkowski inequality if one assumes $a$ and $b$ to have the same sign. If the $a$ and $b$ have different signs, the inequality follows by the fact that $p$ is an even integer and hence the right-hand-side is equal to $|a|^{p-1} + |b|^{p-1}$. Therefore from the generalized Young inequality for $\lambda>0$
\begin{multline*}
|a^{p-1+j}-b^{p-1+j}|\leq  |a^{p-1}-b^{p-1}| |a^{p-1}-b^{p-1}|^{\frac{j}{p-1}}  \\
\leq  \lambda |a^{p-1}-b^{p-1}||a -b | + \frac{C}{ \lambda }  (|a|^{p-2 + 2j}+|b|^{p-2 + 2j})\:,
\end{multline*}
we have for every $\lambda>0$
\begin{multline*}
\sum_{x,y \in \Le} \epsilon^2 \Kg(x-y)\epsilon^{-1}\gamma|\Vg^{p-1+j}(s,x)-\Vg^{p-1+j}(s,y)| \\
\lesssim \lambda\ang{\Vg^{p-1}(s),\Deltag\Vg(s)} + \frac{1}{\lambda}\Norm{L^1}{\Vg^{p+2+2j}(s)}
\end{multline*}
and optimizing in $\lambda$ we get \eqref{e:scalarbound}. Finally we can combine \eqref{e:scalarbound0} and \eqref{e:scalarbound} to conclude the proof.

We remark that the right-hand-side of \eqref{e:scalarbound} is slightly different from \cite{tsatsoulis2016spectral} since we have to use $\Deltag$, the discrete (long range) Laplacian, which is a good approximation of the continuous Laplacian only on low frequencies.
\end{proof}

We now turn to the proof of Theorem~\ref{trm:tightness}, which follows the lines of \cite[Cor.~3.10]{tsatsoulis2016spectral}.
\begin{proof}[of Theorem~\ref{trm:tightness}]
By the monotonicity of $L^q$ norms it is sufficient to prove the statement of Theorem~\ref{trm:tightness} for $q$ large enough. In the following proof $C$ will denote a constant possibly changing from line to line. The Gibbs measure $\Pgbeta$ is an invariant measure for the Glauber dynamic. Fix $T \geq 0$ and let $T/4 \leq s < t \leq T$
\begin{equation}
\label{e:dymdecomp}
\Egbeta\big[\Norm{\Cc^{-\nu}}{\Xg}^q\big] =\frac{2}{T} \int_{T/2}^T \E_{\beta,0}^{\gamma}\big[\Norm{\Cc^{-\nu}}{\Xg(s)}^q\big] ds\;.
\end{equation}
From the definition of $\Vg$ we can write
\[
\Norm{\Cc^{-\nu}}{\Xg(s)} \leq \Norm{\Cc^{-\nu}}{\Zg(s)} + \Norm{\Cc^{-\nu}}{\Vg(s)} 
\]
By \eqref{e:Hermitebound} proven in \cite[Prop.~5.4]{MourratWeber}, we have that
\begin{equation}
\label{e:Zgbesovbounds}
\E_{\beta,0}^{\gamma}\left[ \sup_{s \in [T/4,T]}\Norm{\Cc^{-\nu}}{H_j(\Zg(s),\CG)}^q\right] \leq C(T,q,j)
\end{equation}
where the proportionality constant may depend on $T$ and $q$. From the definition of the discrete Besov norm it follows that \eqref{e:Zgbesovbounds} holds true also when we replace the Besov norm with the discrete Besov Norm. By Proposition~\ref{prop:boundBesovLp} and Lemma~\ref{lemma:Lpdiscretebound}, for any $q > d/\nu$ and $\kappa>0$ there exists $C(p,\kappa)$
\begin{multline}
\label{e:besovbound}
\Norm{\Cc^{-\nu}}{\Vg(s)} \leq \Norm{L^q(\T^2)}{\Ex \Vg(s)} \\ 
\lesssim  \Norm{L^{q}(\Le)}{\Vg(s)} +  \epsilon^{- \kappa} \Norm{L^{2q-2}(\Le)}{\Vg(s)}^{1 - \frac{1}{q}} \Bigg\{ \sum_{\substack{|x-y| = \epsilon\\ x,y \in \Le}} \epsilon^2 (\Vg(s,y)-\Vg(s,x))^2 \Bigg\}^{\frac{1}{2q}}
\end{multline}
where the proportionality constant depends on $q$ and $\kappa$.\\
In Proposition~\ref{prop:Vgcomparison}, using \eqref{e:Zgbesovbounds} and \eqref{e:Vgbounderror} we obtain that
\begin{multline}
\label{e:boundonVgLpoptimal1}
\E_{\beta,0}^{\gamma}\left[\Norm{L^p(\Le)}{\Vg(t)}^p\right] 
+ C_1 \int_s^t \E_{\beta,0}^{\gamma}\left[\Norm{L^p(\Le)}{\Vg(r)}^p\right]^{\frac{p+2}{p}} dr \\
+  C_1\int_s^t \E_{\beta,0}^{\gamma}\left[\ang{\Vg^{p-1}(r),(-\Deltag)\Vg(r)}_{\Le} \right] dr  \leq \E_{\beta,0}^{\gamma}\left[\Norm{L^p(\Le)}{\Vg(s)}^p \right] + C(p,T)
\end{multline}
From Lemma~\ref{lemma:comparison}, applied to $\E_{\beta,0}^{\gamma}\left[\Norm{L^p(\Le)}{\Vg(t)}^p\right]$ we have that there exists $C(p,T)$ such that for all $T/4 \leq s \leq t \leq T$ we have
\[
\E_{\beta,0}^{\gamma}\left[\Norm{L^p(\Le)}{\Vg(t)}^p\right] \lesssim C(p,T) \left(|t-s|^{-\frac{p}{2}} \vee 1 \right)\:.
\]
Let us choose $s=T/4$ and $t \in [T/2,T]$: from the above inequality we have that
\begin{equation}
\label{e:boundonVgLpoptimal2}
\E_{\beta,0}^{\gamma}\left[\Norm{L^p(\Le)}{\Vg(t)}^p\right] \leq C(p,T)\:.
\end{equation}
At this point we only need to provide a bound for
\[
\sum_{\substack{|x-y| = \epsilon\\ x,y \in \Le}} \epsilon^2  (\Vg(s,y)-\Vg(s,x))^2  \lesssim  \epsilon^2   \sum_{\omega \in \LN} |\omega|^2 |\hat{V}_{\gamma}(s)|^2\;.
\]
By \eqref{prop:kernelbounds}, the operator $\Deltag$ approximates the discrete Laplacian only for low frequencies $|\omega|\leq \gamma^{-1}$
\[
| \widehat{\Deltag \Vg}(s)(\omega)| = \gamma^{-2}(1 -\hKg(\omega))| \hat{V}_{\gamma}(s)(\omega)| \geq c  |\omega|^2| \hat{V}_{\gamma}(s)(\omega)|\;.
\]
On the other hand, for high frequencies $\gamma^{-1} \leq |\omega| \leq \gamma^{-2}$, we have
\[
| \widehat{\Deltag \Vg}(s)(\omega)| \geq \gamma^{-2}(1 -\hKg(\omega))| \hat{V}_{\gamma}(s)(\omega)| \gtrsim \gamma^{-2}| \hat{V}_{\gamma}(s)(\omega)|\;,
\]
hence for all $\omega \in \LN$,
\[
 |\omega|^2 |\hat{V}_{\gamma}(s)|^2 \leq \gamma^{-2}(|\omega|^2 \wedge \gamma^{-2}) | \hat{V}_{\gamma}(s)(\omega)|^2 \leq \gamma^{-4}(1 -\hKg(\omega))| \hat{V}_{\gamma}(s)(\omega)|^2
\]
and therefore
\[
\sum_{\omega \in \LN} |\omega|^2 |\hat{V}_{\gamma}(s)|^2 \leq \gamma^{-2} \ang{\Vg(s),(-\Deltag)\Vg(s)}_{\Le}\;.
\]
Using \eqref{e:boundonVgLpoptimal1}, for $s=T/2$, $t=T$ and $p=2$ we conclude that
\begin{multline}
\label{e:boundonVgLpoptimal3}
 \int_{T/2}^T\sum_{|x-y|=\epsilon} \epsilon^2  (\Vg(s,y)-\Vg(s,x))^2 ds\\
  \leq \epsilon^2 \gamma^{-2} \int_{T/2}^T \E_{\beta,0}^{\gamma}\left[\ang{\Vg(r),(-\Deltag)\Vg(r)}_{\Le} \right] dr \leq  C(T) \;.
\end{multline}
It is sufficient now to control the right-hand-side of \eqref{e:dymdecomp} with \eqref{e:besovbound}. By \eqref{e:Zgbesovbounds}, \eqref{e:boundonVgLpoptimal2} and \eqref{e:boundonVgLpoptimal3}
\begin{align*}
\Egbeta\big[\Norm{\Cc^{-\nu}}{\Xg}^q\big] &= \frac{2}{T} \int_{T/2}^T \E_{\beta,0}^{\gamma}\big[\Norm{\Cc^{-\nu}}{\Zg(s)}^q\big] + \E_{\beta,0}^{\gamma}\big[\Norm{\Cc^{-\nu}}{\Vg(s)}^q\big] ds \\
\leq C(T,q,\kappa&) \int_{T/2}^T \E_{\beta,0}^{\gamma}\big[\Norm{\Cc^{-\nu}}{\Zg(s)}^q\big] + \E_{\beta,0}^{\gamma}\Norm{L^q}{\Vg(s)}^q ds \\
+ C(T,q,\kappa&) \E_{\beta,0}^{\gamma} \int_{T/2}^T \Norm{L^{2q-1}}{\Vg(s)}^{q-1} \epsilon^{-q \kappa} \Bigg\{ \sum_{\substack{|x-y| = \epsilon\\ x,y \in \Le}} \epsilon^2 (\Vg(s,y)-\Vg(s,x))^2 \Bigg\}^{\frac{1}{2}} ds \\
\leq C(T,q,\kappa&)\left( 1 +  \epsilon^{-q \kappa} \epsilon \gamma^{-1} \left\{\int_{T/2}^T \E_{\beta,0}^{\gamma}\left[\ang{\Vg(r),(-\Deltag)\Vg(r)}_{\Le} \right] dr\right\}^{1/2} \right)\;,
\end{align*}
where in the last line we applied the Cauchy-Schwarz inequality. The claim follows by choosing $\kappa>0$ small enough.
\end{proof}

\section{Bounds for discrete Besov spaces}
\label{sec:discreteBesov}

In this section we collect and prove some results in the context of discrete Besov spaces which are difficult to find. Let us first define the Besov norm on the discrete torus as follows. The definitions and proofs are based upon \cite{MourratWeber,MourratWeberGlobal} and \cite{bahouri2011fourier}.
In \cite[Prop.~2.10]{bahouri2011fourier} it is proven the existence of continuous functions $\tilde{\chi},\chi : \R^d \to \R$ such that
\begin{equ}
\supp (\tilde{\chi}) \subseteq B_0(4/3)\;,\qquad 
\supp (\chi) \subseteq B_0(8/3) \setminus B_0(3/4)
\end{equ}
and such that, setting
\[
\chi_{-1} \eqdef \tilde{\chi}, \qquad \chi_k(\cdot) \eqdef \chi(2^{-k}\ \cdot ) \qquad \text{ for } (k \geq 0)
\]
one has $\tilde{\chi}(r) + \sum_{k=0}^{\infty} \chi(2^{-k}\ r) = 1$ for all  $r \in \R^d$.

For $g : \T^d \to \R$ define the projection onto the $k$-th Paley-Littlewood block as
\begin{equation}
\label{e:PaleyLittlewoodproj}
\delta_k g(x) = 2^{-d}\sum_{\omega \in \Z^d} \chi_k(\omega)\ \widehat{g}(\omega) e_{\omega}(x)
\end{equation}
for $x \in \T^d$ and $k \geq -1$. 
(The factor $2^{-d}$ is such that $\sum_k \delta_k g = g$, see also \eqref{e:Fourierinversion}.)
Recall that the continuous Besov norm given in \eqref{e:Besov} is then defined in terms of these projections.

We now define a version of the Besov norm for functions defined in the discrete lattice. 
This is obtained by not only extending the function with the extension operator of Section~\ref{subsec:notations}, but also performing the $L^p$ norm in \eqref{e:Besov} on the discrete space $\Le^d$ instead of $\T^d$. Let $f : \Le \to \R$, for $\nu \in \R$, $p,q \in [1,\infty]$,  with $\Norm{\Bb^{\nu}_{p,q}(\Le^d)}{\cdot}$ we define
\begin{equation}
\label{e:Besovdiscretedef}
\Norm{\Bb^{\nu}_{p,q}(\Le^d)}{f} \eqdef \begin{cases}
\left( \sum_{k \geq -1} 2^{\nu k q} \Norm{L^p(\Le^d)}{\delta_k \Ex(f)}^q \right)^{\frac{1}{q}} & \text{ if }q < \infty\\
\sup_{k \geq -1} 2^{\nu k} \Norm{L^p(\Le^d)}{\delta_k \Ex(f)}& \text{ if }q = \infty
\end{cases}
\end{equation}
It is clear, from the definitions of $\Ex(g)$, \eqref{e:PaleyLittlewoodproj} and \eqref{e:etakdef}, that for $x \in \Le^d$
\[
\delta_k\Ex(f)(x) = 2^{-d}\sum_{\omega \in \LN^d} \chi_k(\omega)\ \widehat{f}(\omega) e_{\omega}(x) = \etaN_{k} \ast f(x) \qquad \text{ for } x \in \Le
\]
where $\etaN_{k}(x)$ is defined, for $k \geq -1$ and $x \in \T^d$, by
\begin{equation}
\label{e:etakdef}
\etaN_{k}(x) \eqdef 2^{-d} \sum_{\omega \in \LN^d } \chi_k( \omega) e_{\omega}(x)\;,
\end{equation}
where we abused the notations omitting $N$ from the definition of $\etaN_{k}$.

\begin{remark}
	We could have easily avoided the definition of a discrete version of the Besov norm. 
	There is only one point where such definition is really needed, and this is in Proposition~\ref{prop:Vgcomparison} and Lemma~\ref{lemma:besovregularity}, where we need to control the Besov norm with a combination of discrete $L^p$ norms.
\end{remark}

The next lemma is a minor generalisation of \cite[Lemma~B.6]{MourratWeber}.
\begin{lemma}
\label{lemma:Lpdiscretebound}
For $p \in [1,\infty]$ and $\kappa > 0$, there exists a constant $C$ such that for all $f: \Le \to \R$,
\begin{align}
\label{e:Lpdiscretebound1}
\Norm{L^p(\T^2)}{ \Ex(f)} &\leq C \log^2(\epsilon^{-1}) \Norm{L^p(\Le)}{f} \\
\label{e:Lpdiscretebound2}
\Norm{L^p(\T^2)}{ \Ex(f)} &\leq C\Norm{L^p(\Le)}{ f} + C \epsilon^{ - \kappa} \Norm{L^{2p-2}(\Le)}{f}^{1 - \frac{1}{p}} \Bigg\{ \sum_{\substack{|x-y| = \epsilon\\ x,y \in \Le}}\epsilon^2 (f(y)-f(x))^2 \Bigg\}^{\frac{1}{2p}}\;.
\end{align}
\end{lemma}
The same lemma holds true in any dimension, with a factor $ C(d) \log^d(\epsilon^{-1})$ in \eqref{e:Lpdiscretebound1}.
\begin{proof}
We first show \eqref{e:Lpdiscretebound1}. Recall that from the definition of the extension operator $\Ex f(x) = f(x)$ for $x \in \Le^2$, and 
\[
\Ex(f)(x)= \sum_{z \in \Le} \epsilon^2 f(z) \prod_{j=1,2}\frac{\sin\left(\pi \epsilon^{-1} (x_j - z_j)\right)}{2\sin\left( \frac{\pi}{2}(x_j - z_j)\right)} \qquad x \in \T^2\:.
\]
Using the inequality $\sin(2\epsilon^{-1}a)/\sin(a) \lesssim \epsilon^{-1} \wedge |a|^{-1}$ we can bound
\[
|\Ex(f)(x)| \lesssim \sum_{z \in \Le} \epsilon^2| f(z) |\prod_{j=1,2} \epsilon^{-1} \wedge |z_i-x_i|^{-1}\:.
\]
For $x \in \T^2$, denote with $[x]_{\epsilon}$ the closest point to $x$ in $\Le$. We can then rewrite the above inequality as
\[
|\Ex(f)(x)| \lesssim \sum_{z \in \Le} \epsilon^2 |f(z+[x]_{\epsilon}) |\prod_{j=1,2} \epsilon^{-1} \wedge |z_i+ [x]_{\epsilon,i}-x_i|^{-1}\:.
\]
we observe now that if $|z_i| \leq \epsilon$, then $|z_i+ [x]_{\epsilon,i}-x_i|^{-1} \gtrsim \epsilon^{-1}$, while if $|z_i| > \epsilon$ we have that  $|z_i+ [x]_{\epsilon,i}-x_i|^{-1} \lesssim |z_i|^{-1} \lesssim\epsilon^{-1}$, hence
\[
|\Ex(f)(x)| \lesssim \sum_{z \in \Le} \epsilon^2 |f(z+[x]_{\epsilon})| \prod_{j=1,2} \epsilon^{-1} \wedge |z_i|^{-1}\:,
\]
and taking the $L^p(\T^2,dx)$ norm yields
\[
 \Big(\int_{\T^2}|f( [x]_{\epsilon})|^p dx\Big)^{\frac{1}{p}} \Big( 1 + 2 \sum_{1 \leq k \leq \epsilon^{-1}} k^{-1}\Big)^2 \lesssim \Norm{L^p(\Le)}{f} \log^2(\epsilon^{-1})\;,
\]
as claimed. The inequality \eqref{e:Lpdiscretebound2} is a consequence of H\"older's inequality
\begin{align*}
\Norm{L^p(\T^2)}{ \Ex(f)}^p &\lesssim \Norm{L^p(\Le)}{f}^p + \int_{|y|\leq \epsilon/2}  |\Ex f(x+y) - f(x)|^p d^2y\\
\leq &\Norm{L^p(\Le)}{f}^p+ \Norm{L^{2p-2}(\T^2)}{\Ex f}^{p-1} \left(\sum_{x \in \Le} \int_{|y|\leq \epsilon/2} |f(x+y) - f(x)|^2 d^2y\right)^{\frac{1}{2}}\\
\leq &\Norm{L^p(\Le)}{f}^p+ \Norm{L^{2p-2}(\T^2)}{\Ex f}^{p-1} \left(\sum_{x \in \Le} \int_{|y|\leq \epsilon/2} |f(x+y) - f(x)|^2 d^2y\right)^{\frac{1}{2}}\\
\leq &\Norm{L^p(\Le)}{f}^p+ \epsilon  \Norm{L^{2p-2}(\T^2)}{\Ex f}^{p-1} \Norm{\dot{H}^1(\T^2)}{\Ex f}
\end{align*}
where we denoted by $\Norm{\dot{H}^1(\T^2)}{\Ex f}$ the homogeneous Sobolev seminorm. From the definition of the extension operator it is easy to see that
\[
\Norm{\dot{H}^1(\T^2)}{\Ex f}^2 = \sum_{\omega \in \LN} |\omega|^2|\hat{f}(\omega)|^2 \lesssim  \sum_{\substack{|x-y| = \epsilon\\ x,y \in \Le}} \epsilon^2 \frac{(f(y)-f(x))^2}{\epsilon^2} \;,
\]
and an application of \eqref{e:Lpdiscretebound1} yields \eqref{e:Lpdiscretebound2}.
\end{proof}

Let $d \in \N^+$ and $\Le^d$ a discretisation of the $d$-dimensional torus $\T^d=[-1,1]^d$.

\begin{lemma}
\label{lemma:etakbound}
Let $\chi: \R^d \to \R$ be a smooth function with compact support. For every $p \in [0,\infty]$ we have
$
\sup_{\lambda \in (0,1)} \lambda^{d\left( 1-\frac{1}{p}\right)} \Norm{L^p(\Le^d)}{\sum_{w \in \LN } \chi(\lambda w) e_w}< \infty\;.
$
\end{lemma}

The above result is proven in \cite[Lem.~B.1]{MourratWeber} for the $L^p$ in the whole torus, the generalisation to $\Le^d$ follows trivially from the same argument. See also \cite{MourratWeberGlobal} for a proof in the case of a continuous Fourier transform.
We quote in the next proposition a useful embedding between Besov and $L^p$ spaces proven, for instance, in \cite[Prop.~2.39]{bahouri2011fourier}.

\begin{proposition}
\label{prop:boundBesovLp}
For any $\nu > 0$, and $p \geq \frac{d}{\nu}$ there exists $C>0$
\[
\Norm{\Bb^{-\nu}_{\infty,\infty}(\T^d)}{f} \leq C \Norm{L^p(\T^d)}{f}
\]
\end{proposition}

We also mention some classical estimates for Besov norms restated in case of discrete Besov spaces. The proofs are omitted since they follow closely their continuous counterparts.

\begin{proposition}[Product estimates for discrete Besov spaces]
Let $\beta < 0 < \alpha$ and $p,q \in [1,\infty]$. There exists $C>0$ such that, uniformly over $\eps \in (0,1]$,
\[
\Norm{\Bb^{\beta}_{p,q}(\Le)}{f}\leq C\Norm{\Bb^{\alpha}_{p,q}(\Le)}{f}\Norm{\Bb^{\beta}_{p,q}(\Le)}{g}\;.
\]
\end{proposition}

\begin{proposition}[Duality for discrete Besov spaces]
\label{prop:discretedualityBesov}
Let $\alpha \in \R$, $p,q,p',q'\geq 1$ with $\frac{1}{p}+\frac{1}{p'} = \frac{1}{q}+\frac{1}{q'}=1$. There exists $C>0$ such that, uniformly over $\eps \in (0,1]$,
\[
\ang{f,g}_{\Le} \leq C \Norm{\Bb^{\alpha}_{p,q}(\Le)}{f} \Norm{\Bb^{-\alpha}_{p',q'}(\Le)}{g}\;.
\]
\end{proposition}

The next proposition is the main technical tool of the paper, and it allows to control the discrete Besov norm with the same discrete Laplacian of the dynamic. Recall that, in the case of continuous Besov spaces, for a differentiable function $f$ and $\nu \in (0,1)$, one has \cite[Prop.~3.8]{MourratWeberGlobal}
\begin{equation}
\label{e:besovderivboundcontinuous}
\Norm{\Bb^{\nu}_{1,1}}{ f } \lesssim \Norm{L^1}{f} + \Norm{L^1}{f}^{1-\nu}\Norm{L^1}{\nabla f}^{\nu}\;.
\end{equation}
We have the following analogue of this result.

\begin{lemma}
\label{lemma:besovregularity}
For $f: \Le \to \R$ and $\nu \in (0,1/2)$
\begin{equation}
\label{e:besovderivbound}
\Norm{\Bb_{1,1}^{\nu}(\Le)}{f} \lesssim \Norm{L^1(\Le)}{f}^{1-2\nu}  \Big(\sum_{x,y \in \Le} \epsilon^4 \Kg(x-y)\epsilon^{-1}\gamma|f(x)-f(y)| \Big)^{2\nu}+  \Norm{L^1(\Le)}{f}
\end{equation}
where the constant is independent of $\epsilon$ or $f$. 
\end{lemma}

\begin{remark}
Here we write $x-y$ for the shortest element in the corresponding equivalence class
(viewing $\Lambda_\eps$ as a quotient of $\Ze^2$ with $\Ze = \eps \Z$). In cases where this might be ambiguous one
has $\Kg(x-y)=0$ for all possible interpretations anyway.
Another equivalent interpretation is that one of the variables runs over $\Le$ and the other one runs over all of $\Ze$, $f$ being identified with its periodic continuation.
\end{remark}

Compare \eqref{e:besovderivbound} with \eqref{e:besovderivboundcontinuous}. The factor $2$ in front of 
$\nu$ depends on the scale at which $\Deltag$ changes its behaviour, and this is not the best 
result that is possible to obtain.

\begin{proof}
Rewrite the definition of $\Norm{\Bb^{\nu}_{1,1}(\Le)}{f}$ in \eqref{e:Besovdiscretedef} as
\begin{equ}[e:startingPoint]
\Norm{\Bb^{\nu}_{1,1}(\Le)}{f} = \sum_{k \geq -1} 2^{\nu k} \Norm{L^1(\Le)}{\etaN_k \ast f}
\end{equ}
where $\etaN_k$ are the projections on the Paley-Littlewood blocks defined in \eqref{e:etakdef}. In the discrete case the summation over $k$ extends up to a multiple of $\log(\epsilon^{-1})$. In the proof, since there is no possibility of confusion, we will use $L^p$ instead of $L^p(\Le)$. We will divide the sum into
\[
\sum_{-1 \leq k \leq L} 2^{\nu k} \Norm{L^1}{\etaN_k \ast f} + \sum_{L < k \leq -\log_2(\epsilon)} 2^{\nu k} \Norm{L^1}{\etaN_k \ast f}
\]
where $L$ will be chosen later. We bound the first part with
\begin{equ}[e:bound1]
\sum_{-1 \leq k \leq L} 2^{\nu k} \Norm{L^1}{\etaN_k \ast f} \leq \sum_{-1 \leq k \leq L} 2^{\nu k} \sup_{k'\le L} \Norm{L^1}{\etaN_{k'}} \Norm{L^1}{ f} \lesssim 2^{\nu L} \Norm{L^1}{ f}\:.
\end{equ}
In order to control the second summation we will now prove, for $k \geq 0$, the inequality
\begin{equ}[e:wantedIntermediate]
\Norm{L^1}{\etaN_k \ast f} \lesssim \left( 2^{-k} \vee \epsilon \gamma^{-1} \right)\sum_{x,y \in \Le} \epsilon^4 \Kg(x-y) \frac{|f(y) -f(x)|}{\epsilon\gamma^{-1}}\;.
\end{equ}
If $k \geq 0$ the projection kernel $\etaN_k$ has mean zero and therefore
\[
\Norm{L^1}{\etaN_k \ast f} = \sum_{x \in \Le} \epsilon^2 \left|\sum_{y \in \Le} \epsilon^2 \etaN_k(-y) \Big(f(x+y) - f(x)\Big)\right|\:.
\]
At this point the treatment differs from the proof of \cite[Prop.~3.8]{MourratWeberGlobal}, because of the particular form of the Laplacian. The definition of $\Kg$ (in particular the continuity of $\KK$) implies that 
there exists $b_0 > 0$ such that
\[
\inf_{ |w|\leq b_0 \epsilon\gamma^{-1}} \sum_{ z \in \Ze}\epsilon^2 \big(\Kg(z) \wedge \Kg(w-z)\big) \geq 1/2\:.
\]
If $|y| \leq b_0 \epsilon\gamma^{-1}$, then
\begin{multline*}
\Big|f(x+y) - f(x)\Big| \leq 2 \Big( \sum_{ z \in \Ze}\epsilon^2 \big(\Kg(z) \wedge \Kg(y-z)\big)  \Big) \Big|f(x+y) - f(x)\Big|\\
\leq 2\epsilon^2 \sum_{ z \in \Ze } \Kg(y-z)  \Big|f(x+y) - f(x+z)\Big| + \Kg(z)  \Big|f(x+z) - f(x)\Big|\:.
\end{multline*}
If $|y| \geq b_0 \epsilon\gamma^{-1}$ on the other hand, then there exists a path $\{y_0,y_1,\dots,y_n\}$ in $\Ze$ of length $n$ proportional to $|y|\gamma \epsilon^{-1}$ connecting $y_0 = 0$ with $y_n = y$ and such that $|y_{j+1}-y_j| \leq b_0 \epsilon \gamma^{-1}$ for $j= 0,\dots,n-1$. We can then apply the above inequality to every step of the path. Combining these bounds, we obtain
\begin{equ}[e1]
\Norm{L^1}{\etaN_k \ast f} \lesssim \sum_{y \in \Le} \epsilon^2 |\etaN_k(-y)| \{ |y|\gamma \epsilon^{-1} \vee 1 \} \sum_{x \in \Le, z  \in \Ze} \epsilon^4 \Kg(z )  |f(x+z) -f(x)|\;,
\end{equ}
and \eqref{e:wantedIntermediate} follows from the fact that
\[
\sum_{y \in \Le} \epsilon^2 |\etaN_k( y)| \{ |y|\epsilon^{-1} \gamma \vee  \gamma \} \lesssim   \Norm{L^1}{\etaN_k} + \epsilon^{-1} \gamma 2^{-k} \sum_{y \in \Le}2^k|y||\etaN_k(y)| \lesssim 1 \vee \epsilon^{-1} \gamma 2^{-k}\:.
\]
Summing over $k$ yields
\begin{equs}
\sum_{L < k \leq  \log_2(\epsilon^{-1})} 2^{\nu k} \Norm{L^1}{\etaN_k \ast f}& \\
\lesssim  \sum_{L < k \leq \log_2(\epsilon^{-1})}& 2^{\nu k}\{ \epsilon \gamma^{-1} \vee 2^{-k}\} \sum_{x  \in \Le, z\in \Ze} {\epsilon^4 \Kg(z ) \over \epsilon \gamma^{-1}}  |f(x+z) -f(x)|\;.
\end{equs}
At this point we use the fact that $\epsilon \gamma^{-1} \vee 2^{-k} \leq 2^{-\frac{k}{2}}$ for $k \leq -\log_2(\epsilon) = -2\log_2(\epsilon\gamma^{-1})$ and, recalling \eqref{e:startingPoint} and \eqref{e:bound1}, 
we obtain
\[
\Norm{\Bb^{\nu}_{1,1}(\Le)}{f}  \lesssim 2^{\nu L} \Norm{L^1}{f}  + 2^{\left(\nu-\frac{1}{2}\right)L}\sum_{x  \in \Le, z\in \Ze} {\epsilon^3\gamma \Kg(z )} |f(x+z) -f(x)|\;.
\]
The claim now follows by optimising this expression over $L$.
(The second term in \eqref{e:besovderivbound} comes from the fact that we had to impose $L > 1$.)
\end{proof}

The next proposition quantifies the decay in the Fourier space of the kernel $\Kg$ used in the article. The proof is given in \cite[Lem.~8.2]{MourratWeber}

\begin{proposition}[estimates on the kernel]
\label{prop:kernelbounds}
For $\omega \in \LN$ and $\gamma$ small enough the following inequalities hold
\begin{itemize}
\item There exists a positive constant $C>0$ such that,
\[
|\hKg(\omega)|\leq 1  \wedge \frac{\gamma^{-2}}{|\omega|^2}
\]
\item There exists a positive constant $c>0$ such that, for $|\omega| \geq \gamma^{-1}$
\[
1 - \hKg(\omega) \geq c \left(|\gamma \omega|^2\wedge 1\right)
\]
\end{itemize}
\end{proposition}

For completeness, we state the following simple but crucial comparison test which can be found in this specific form in
\cite{tsatsoulis2016spectral}.

\begin{lemma}[Comparison test]
\label{lemma:comparison}
Let $\lambda > 1$ and $f:[0,T] \to \R^+$ differentiable satisfying for $t \in [0,T]$
\[
f'(t) + 2 c_1 \left(f(t) \right)^{\lambda}\leq c_2\;.
\]
Then for $t \in [0,T]$
\[
f(t) \leq \frac{f(0)}{\left( 1 +  c_1 (\lambda-1) t f(0)^{\lambda-1}\right)^{\frac{1}{\lambda-1}}} \vee \left( \frac{c_2}{c_1}\right)^{\frac{1}{\lambda}} \leq \left(c_1(\lambda-1) t\right)^{-\frac{1}{\lambda-1}} \vee \left( \frac{c_2}{c_1}\right)^{\frac{1}{\lambda}}\;.
\]
\end{lemma}

\bibliographystyle{Martin}
\bibliography{references}

\begin{thebibliography}{BPRS93}
\expandafter\ifx\csname url\endcsname\relax
  \def\url#1{\texttt{#1}}\fi
\expandafter\ifx\csname urlprefix\endcsname\relax\def\urlprefix{URL }\fi

\bibitem[BCD11]{bahouri2011fourier}
\textsc{H.~Bahouri}, \textsc{J.-Y. Chemin}, and \textsc{R.~Danchin}.
\newblock \emph{Fourier analysis and nonlinear partial differential equations},
  vol. 343 of \emph{Grundlehren der Mathematischen Wissenschaften [Fundamental
  Principles of Mathematical Sciences]}.
\newblock Springer, Heidelberg, 2011.

\bibitem[BPRS93]{MR1317994}
\textsc{L.~Bertini}, \textsc{E.~Presutti}, \textsc{B.~R{\"u}diger}, and
  \textsc{E.~Saada}.
\newblock Dynamical fluctuations at the critical point: convergence to a
  nonlinear stochastic {PDE}.
\newblock \emph{Teor. Veroyatnost. i Primenen.} \textbf{38}, no.~4, (1993),
  689--741.
\newblock \ifx\href\undefined
  \texttt{doi:10.1137/1138062}\else\href{http://dx.doi.org/10.1137/1138062}{\texttt{doi:10.1137/1138062}}\fi.

\bibitem[CMP95]{cassandro1995corrections}
\textsc{M.~Cassandro}, \textsc{R.~Marra}, and \textsc{E.~Presutti}.
\newblock Corrections to the critical temperature in 2d ising systems with kac
  potentials.
\newblock \emph{Journal of statistical physics} \textbf{78}, no.~3, (1995),
  1131--1138.

\bibitem[DPD03]{dPD}
\textsc{G.~Da~Prato} and \textsc{A.~Debussche}.
\newblock Strong solutions to the stochastic quantization equations.
\newblock \emph{Ann. Probab.} \textbf{31}, no.~4, (2003), 1900--1916.

\bibitem[FR95]{Fritz1995}
\textsc{J.~Fritz} and \textsc{B.~R\"udiger}.
\newblock Time dependent critical fluctuations of a one-dimensional local mean
  field model.
\newblock \emph{Probab. Theory Related Fields} \textbf{103}, no.~3, (1995),
  381--407.
\newblock \ifx\href\undefined
  \texttt{doi:10.1007/BF01195480}\else\href{http://dx.doi.org/10.1007/BF01195480}{\texttt{doi:10.1007/BF01195480}}\fi.

\bibitem[GeK85]{Gawedzki1985}
\textsc{K.~Gaw\k~edzki} and \textsc{A.~Kupiainen}.
\newblock Massless lattice {$\varphi^4_4$} theory: rigorous control of a
  renormalizable asymptotically free model.
\newblock \emph{Comm. Math. Phys.} \textbf{99}, no.~2, (1985), 197--252.

\bibitem[GJ87]{GlimmJaffe}
\textsc{J.~Glimm} and \textsc{A.~Jaffe}.
\newblock \emph{Quantum physics}.
\newblock Springer-Verlag, New York, second ed., 1987.
\newblock A functional integral point of view.

\bibitem[KUH63]{KUH1963vanderWaals}
\textsc{M.~Kac}, \textsc{G.~E. Uhlenbeck}, and \textsc{P.~C. Hemmer}.
\newblock On the van der {W}aals theory of the vapor-liquid equilibrium. {I}.
  {D}iscussion of a one-dimensional model.
\newblock \emph{J. Mathematical Phys.} \textbf{4}, (1963), 216--228.
\newblock \ifx\href\undefined
  \texttt{doi:10.1063/1.1703946}\else\href{http://dx.doi.org/10.1063/1.1703946}{\texttt{doi:10.1063/1.1703946}}\fi.

\bibitem[Leb74]{Lebowitz1974inequalities}
\textsc{J.~L. Lebowitz}.
\newblock G{HS} and other inequalities.
\newblock \emph{Comm. Math. Phys.} \textbf{35}, (1974), 87--92.

\bibitem[LP66]{LebowitzPenrose1966}
\textsc{J.~L. {Lebowitz}} and \textsc{O.~{Penrose}}.
\newblock {Rigorous Treatment of the Van Der Waals-Maxwell Theory of the
  Liquid-Vapor Transition}.
\newblock \emph{Journal of Mathematical Physics} \textbf{7}, (1966), 98--113.
\newblock \ifx\href\undefined
  \texttt{doi:10.1063/1.1704821}\else\href{http://dx.doi.org/10.1063/1.1704821}{\texttt{doi:10.1063/1.1704821}}\fi.

\bibitem[MW17a]{MourratWeber}
\textsc{J.-C. Mourrat} and \textsc{H.~Weber}.
\newblock Convergence of the two-dimensional dynamic {I}sing-{K}ac model to
  {$\Phi^4_2$}.
\newblock \emph{Comm. Pure Appl. Math.} \textbf{70}, no.~4, (2017), 717--812.
\newblock \ifx\href\undefined
  \texttt{doi:10.1002/cpa.21655}\else\href{http://dx.doi.org/10.1002/cpa.21655}{\texttt{doi:10.1002/cpa.21655}}\fi.

\bibitem[MW17b]{MourratWeberGlobal}
\textsc{J.-C. Mourrat} and \textsc{H.~Weber}.
\newblock Global well-posedness of the dynamic {$\Phi^4$} model in the plane.
\newblock \emph{Ann. Probab.} \textbf{45}, no.~4, (2017), 2398--2476.

\bibitem[Nel66]{Nelson}
\textsc{E.~Nelson}.
\newblock A quartic interaction in two dimensions.
\newblock In \emph{Mathematical {T}heory of {E}lementary {P}articles ({P}roc.
  {C}onf., {D}edham, {M}ass., 1965)},  69--73. M.I.T. Press, Cambridge, Mass.,
  1966.

\bibitem[Pre09]{presutti2008scaling}
\textsc{E.~Presutti}.
\newblock \emph{Scaling limits in statistical mechanics and microstructures in
  continuum mechanics}.
\newblock Theoretical and Mathematical Physics. Springer, Berlin, 2009.

\bibitem[SG73]{Simon1973}
\textsc{B.~Simon} and \textsc{R.~B. Griffiths}.
\newblock The {$(\phi ^{4})_{2}$} field theory as a classical {I}sing model.
\newblock \emph{Comm. Math. Phys.} \textbf{33}, (1973), 145--164.

\bibitem[SW16]{ShenWeber}
\textsc{H.~Shen} and \textsc{H.~Weber}.
\newblock Glauber dynamics of 2d kac-blume-capel model and their stochastic pde
  limits.
\newblock \emph{arXiv preprint arXiv:1608.06556} (2016).

\bibitem[TW16]{tsatsoulis2016spectral}
\textsc{P.~Tsatsoulis} and \textsc{H.~Weber}.
\newblock Spectral gap for the stochastic quantization equation on the
  2-dimensional torus.
\newblock \emph{arXiv preprint arXiv:1609.08447} (2016).

\end{thebibliography}

\end{document}